\documentclass[11pt]{article}
\usepackage{latexsym,amsmath,amssymb,paralist,subfigure,graphicx}
\usepackage{multirow}
\usepackage{comment}
\usepackage[compress]{cite}
\usepackage{algorithm}
\usepackage{algpseudocode}
\usepackage[title]{appendix}
\usepackage{xcolor}

\usepackage{graphicx}
\usepackage{epstopdf}
\usepackage{bm}
    \usepackage{nicefrac}
    \usepackage{longtable}
\usepackage{hyperref}
\usepackage[normalem]{ulem}

\newcommand{\reff}[1]{{\rm (\ref{#1})}}

\newcommand{\nabh}{\nabla_{\! h}}

\newcommand{\calC}{\mathcal C}

\newcommand{\hf}{\nicefrac{1}{2}}
\newcommand{\nrm}[1]{\left\| #1 \right\|}
\newcommand{\ciptwo}[2]{\left\langle #1 , #2 \right\rangle}

\newcommand\dt {{\Delta t}}

\newcommand{\eipx}[2]{\left[ #1 , #2 \right]_{\rm x}}
\newcommand{\eipy}[2]{\left[ #1 , #2 \right]_{\rm y}}
\newcommand{\eipz}[2]{\left[ #1 , #2 \right]_{\rm z}}
\newcommand{\eipvec}[2]{\left[ #1 , #2 \right]}

\newtheorem{theorem}{Theorem}[section]
\newtheorem{lemma}[theorem]{Lemma}
\newtheorem{proposition}[theorem]{Proposition}
\newtheorem{corollary}[theorem]{Corollary}
\newtheorem{remark}[theorem]{Remark}

\newenvironment{proof}[1][Proof]{\begin{trivlist}
\item[\hskip \labelsep {\bfseries #1}]}{\end{trivlist}}

\newcommand{\qed}{\nobreak \ifvmode \relax \else
      \ifdim\lastskip<1.5em \hskip-\lastskip
      \hskip1.5em plus0em minus0.5em \fi \nobreak
      \vrule height0.75em width0.5em depth0.25em\fi}

\def\XXint#1#2#3{{\setbox0=\hbox{$#1{#2#3}{\int}$}
\vcenter{\hbox{$#2#3$}}\kern-.51\wd0}}

\newcommand{\td}{\tilde}

{\allowdisplaybreaks

\begin{document}

\graphicspath{{Figures/}}

\title{Convergence Analysis of Structure-Preserving Numerical Methods Based on Slotboom Transformation for the Poisson--Nernst--Planck Equations}

\author{Jie Ding\thanks{
School of Science, Jiangnan University,  Wuxi, Jiangsu, 214122, China (jding@jiangnan.edu.cn).}
\and
Cheng Wang\thanks{Department of Mathematics, University of Massachusetts, North Dartmouth, MA  02747, USA (cwang1@umassd.edu).}
\and	
Shenggao Zhou\thanks{School of Mathematical Sciences, MOE-LSC, and CMA-Shanghai, Shanghai Jiao Tong University, Shanghai, 200240, China  (Corresponding Author: sgzhou@sjtu.edu.cn).}
}
\maketitle

\begin{abstract}
The analysis of structure-preserving numerical methods for the Poisson--Nernst--Planck (PNP) system has attracted growing interests in recent years. A class of numerical algorithms have been developed based on the Slotboom reformulation, and the mass conservation, ionic concentration positivity, free-energy dissipation have been proved at a discrete level. Nonetheless, a rigorous convergence analysis for these Slotboom reformulation-based, structure-preserving schemes has been an open problem for a long time.  In this work, we provide an optimal rate convergence analysis and error estimate for finite difference schemes based on the Slotboom reformulation.  Different options of mobility average at the staggered mesh points are considered in the finite-difference spatial discretization, such as the harmonic mean, geometric mean, arithmetic mean, and entropic mean. A semi-implicit temporal discretization is applied, which in turn results in a non-constant coefficient, positive-definite linear system at each time step.  A higher order asymptotic expansion is applied in the consistency analysis, and such a higher order consistency estimate is necessary to control the discrete maximum norm of the concentration variables. In convergence estimate, the harmonic mean for the mobility average, which turns out to bring lots of convenience in the theoretical analysis, is taken for simplicity, while other options of mobility average would also lead to the desired error estimate, with more technical details involved.  As a result, an optimal rate convergence analysis on concentrations, electric potential, and ionic fluxes is derived, which is the first such results for the structure-preserving numerical schemes based on the Slotboom reformulation. It is remarked that the convergence analysis leads to a theoretical justification of the conditional energy dissipation analysis, which relies on the maximum norm bounds of the concentration and the gradient of the electric potential. Some numerical results are also presented to demonstrate the accuracy and structure-preserving performance of the associated schemes.

\end{abstract}
{\bf Keywords:} Poisson--Nernst--Planck equations, Slotboom reformulation, mobility average, convergence analysis and error estimate,  higher order consistency estimate

\section{Introduction}
The Poisson--Nernst--Planck (PNP) equations are widely used to describe charge transport in many applications, e.g., transmembrane ion channels~\cite{IonChanel_HandbookCRC15}, electrochemical devices~\cite{BTA:PRE:04}, and semiconductors~\cite{PMarkowich_Book}. Such a system of equations can be formulated in a dimensionless form as 
\begin{equation}\label{PNP}
\left\{
\begin{aligned}
&\partial_t c^l= \nabla\cdot(\nabla c^l+q^lc^l\nabla\psi) , \quad  l=1, \cdots, M, \qquad \mbox{in }\Omega, \\
&-\kappa \Delta \psi =  \sum_{l=1} ^M q^l c^l + \rho^f  ,  \qquad \mbox{in }\Omega, 
\end{aligned}
\right.
\end{equation}
where $c^l=c^l({\bf x}, t)$ and $q^l$ are the ionic concentration and valence for the $l$-th species, $\kappa$ stands for the coefficient arising from nondimensionalization, $\psi=\psi({\bf x}, t)$ is the electric potential, and $\rho^f=\rho^f({\bf x})$ is the distribution of fixed charge density. For simplicity of presentation, we consider a cubic $\Omega$ with periodic boundary conditions.  An extension to the model with homogeneous Neumann boundary conditions is straightforward. The initial condition for ionic concentrations is given by $c^l({\bf x}, 0)=c^l_{\rm in}({\bf x})$, where $c^l_{\rm in}$ is the initial data. With periodic boundary conditions, the total mass of ions is conservative in $\Omega$ for each species.

The evolution of ionic concentrations and electric potential described by the PNP equations can be regarded as a gradient flow of an electrostatic free energy
\begin{equation}\label{F}
 F = \sum_{l=1}^M  \int_{\Omega} \left[c^l\log c^l+\frac{1}{2} (q^lc^l+\rho^f)\psi\right]dV.
\end{equation}
In fact, one can derive the free-energy dissipation law
\begin{equation}\label{dF/dtLaw}
\begin{aligned}
\frac{dF}{dt}=\sum_{l=1}^M \int_{\Omega} \frac{\partial c^l}{\partial t} \frac{\delta F}{\delta c^l} dV = -\sum_{l=1}^M \int_{\Omega} \frac{1}{c^l}\big|\nabla c^l+ q^lc^l\nabla\psi\big|^2dV\leq 0. 
\end{aligned}
\end{equation}


Much effort has been devoted to the development of numerical methods for the PNP equations, such as finite difference, finite volume, and finite element methods~\cite{ProhlSchmuck09, LHMZ10, ZCW11, AMEKLL14, CCAO14, LW14, Chatard2014, Gibou_JCP14, LiuShu_SCM16, SunSunZhengLin16,  MXL16, GaoHe_JSC17, LW17, DSWZhou_CICP18, SWZ_CMS18, QianWangZhou_JCP2019, DingWangZhou_NMTMA19, HuHuang_NM20, ShenXu_NM21, LiuMaimaiti2021}. Many existing works focus on the development of numerical algorithms that are able to preserve desirable physical structure properties, including mass conservation, free-energy dissipation, and positivity of ionic concentrations at discrete level. For instance, a second-order accurate, mass conservative, and energy dissipative finite difference method was proposed for the 1-D PNP equations~\cite{LW14}.  A finite element scheme that ensures positivity of ionic concentration via a variable transformation was developed for the PNP-type equations~\cite{MXL16}. A discontinuous Galerkin method was also developed~\cite{LW17}, in which the positivity of the numerical solution was not theoretically proved, while such a property was enforced by an accuracy-preserving limiter. An implicit finite difference method that guarantees numerical positivity and energy dissipation was developed in the work~\cite{HuHuang_NM20}. A set of numerical schemes that unconditionally ensure positivity, unique solvability, and energy dissipation was proposed in~\cite{ShenXu_NM21}. Recently, a mass conservative, positivity-preserving, and energy stable finite difference scheme, based on a gradient-flow formulation, was proposed and analyzed, and an optimal rate convergence analysis has also been established~\cite{LiuC2021a}. Keeping the structure-preserving properties, it has been extended to second-order accuracy, both in space and time, with corresponding convergence analysis in a more recent work~\cite{LiuWang_Sub21}.


The above-mentioned progress on structure-preserving numerical methods can be sorted into two categories. One is based on the observation that the dynamics described by the PNP equations is the gradient flow of the free energy~\reff{F}~\cite{QianWangZhou_JCP21, ShenXu_NM21, LiuC2021a}. Implicit discretization based on the gradient-flow structure naturally respects energy dissipation due to the convexity of the free energy. The existence of positive concentration solution is enforced by using the singularity of the logarithmic function at zero. Convergence analysis and error estimate have been rigorously established for first-order and second-order temporal discretization schemes~\cite{LiuC2021a, LiuWang_Sub21}. The other category is based on the Slotboom transformation~\cite{LW14, SWZ_CMS18, HuHuang_NM20, DingWangZhou_JCP2019, DingWangZhou_JCP2020, LiuMaimaiti2021}, which converts the Nernst--Planck (NP) equations into 
\begin{equation} \label{NPe} 
\begin{aligned} 
  & 
\partial_t c^l=\nabla\cdot(e^{-S^l}\nabla g^l),  \quad S^l= q^l \psi,  \quad 
\quad  l=1, \cdots, M.
\end{aligned}   
\end{equation}
Here $g^l=c^le^{S^l}$ is a generalized Slotboom variable~\cite{Slotboom1973}. The advantage of such a reformulation is associated with the combination of the diffusion and convection into a self-adjoint operator, which greatly facilitates the design of discretization schemes that are able to preserve the discrete maximum principle. Conditional free-energy dissipation can be achieved if the electric potential is explicitly treated, and unconditional dissipation can be proved with an implicit treatment of the electric potential. However, rigorous convergence analysis as well as error estimate is still an open problem for structure-preserving numerical methods based on the Slotboom transformation. 

In this work, we provide a theoretical proof of convergence analysis and error estimate for structure-preserving numerical schemes based on the Slotboom transformation.  After the transformation, the PNP equations are spatially discretized with central differencing approximations, and a semi-implicit treatment is considered for temporal discretization. This in turn results in a non-constant coefficient, positive-definite linear system at each time step. Moreover, a discrete average of mobility functions is needed at the staggered mesh points, and four different options are considered: the harmonic mean, geometric mean, arithmetic mean, and entropic mean. The mass conservative, positivity-preserving, and energy dissipative properties of the finite difference scheme with various mobility averages have already been established in a few recent works~\cite{LW14, DingWangZhou_JCP2019, DingWangZhou_JCP2020, LiuMaimaiti2021}.  In particular, the energy dissipation analysis relies on the maximum norm bounds of the concentration, as well as the gradient of the electric potential, while these bounds have to be established through an optimal rate convergence analysis. We perform the error estimate by employing a higher order asymptotic expansion in the consistency analysis, which is necessary to control the discrete maximum norm of the concentration variables. This approach has been reported for a wide class of nonlinear PDEs; see the related works~\cite{E95, E02, STWW03, STWW07, WL00, WL02, WLJ04, baskaran13b, guan14a, LiX2021}. In the convergence estimate, the harmonic mean for the mobility average, which turns out to bring lots of convenience in the theoretical analysis, is shown as an example, while other options of mobility average would also lead to the desired error estimate, with more technical details involved. As a result, an optimal rate convergence analysis is derived for ionic concentrations, electric potential, and fluxes, which is the first such result for structure-preserving numerical methods for the PNP equations that are based on the Slotboom transformation. 

The rest of the paper is organized as follows. In Section~\ref{sec:numerical scheme}, we present the fully discrete finite difference numerical scheme based on the Slotboom reformulation. The mass conservative, positivity-preserving, and energy dissipative properties of the numerical schemes are recalled in Section~\ref{sec:Properties}. Subsequently, the optimal rate convergence analysis is presented in Section~\ref{sec:convergence}.  Some numerical results are provided in Section~\ref{sec:numerical results}. Finally, conclusions and discussions are given in Section~\ref{sec:conclusion}.

\section{Numerical schemes} \label{sec:numerical scheme}
\subsection{Discretization and notations}\label{ss:Discretization}
The PNP system~\reff{PNP} is discretized based on the Slotboom transformation~\reff{NPe}. For simplicity, a rectangular computational domain $\Omega= (a, b)^3$ with periodic boundary conditions is considered. Let $N\in \mathbb{N}^*$ be the number of grid points along each dimension, and $h=(b-a)/N$ be the uniform grid spacing. The computational domain is covered by the grid points
\[
\left\{x_i, y_j, z_k\right\}= \left\{a+(i-\frac{1}{2})h, a+(j-\frac{1}{2})h, a+(k-\frac{1}{2})h\right\} , 
\]
for $i,j,k= 1, \cdots, N$.
Denote by $c^{l}_{i,j,k}$, $g^{l}_{i,j,k}$, and $\psi_{i,j,k}$ the discrete approximations of $c^l(x_i,y_j,z_k, \cdot)$, $g^l(x_i,y_j,z_k, \cdot)$, and $\psi(x_i,y_j,z_k, \cdot)$, respectively.

We recall the following standard discrete operators and notations for finite difference discretization~\cite{wise09a, wang11a}. The following space of periodic grid functions is introduced:
\begin{equation}
\begin{aligned}
&\mathcal{C}_{\rm per}:=\big\{u\big|u_{i,j,k}=u_{i+\alpha N,j+\beta N,k+\gamma N}, \forall i,j,k, \alpha,\beta,\gamma\in \mathbb{Z} \big\},\\
&{\mathcal E}^{\rm x}_{\rm per}:=\left\{\nu  \middle| \ \nu_{i+\frac12,j,k}= \nu_{i+\frac12+\alpha N,j+\beta N, k+\gamma N}, \ \forall \, i,j,k,\alpha,\beta,\gamma\in \mathbb{Z}\right\},
\end{aligned}
\end{equation}
and 
\[
\mathring{\mathcal{C}}_{\rm per}:=\bigg\{u\in \mathcal{C}_{per}\bigg| \overline{u}=0 \bigg\}, \quad 
 \overline{u}=\frac{h^3}{|\Omega|}\sum_{i,j,k=1}^N u_{i,j,k} . 
\]
Analogously, we define the spaces ${\mathcal E}^{\rm y}_{\rm per}$ and ${\mathcal E}^{\rm z}_{\rm per}$, and denote $\vec{\mathcal{E}}_{\rm per} := {\mathcal E}^{\rm x}_{\rm per}\times {\mathcal E}^{\rm y}_{\rm per}\times {\mathcal E}^{\rm z}_{\rm per}$. We also introduce the following average and difference operators in $x$-direction:  
	\begin{eqnarray*}
&& A_x f_{i+\hf,j,k} := \frac{1}{2}\left(f_{i+1,j,k} + f_{i,j,k} \right), \quad D_x f_{i+\hf,j,k} := \frac{1}{h}\left(f_{i+1,j,k} - f_{i,j,k} \right), \\
&& a_x f_{i, j, k} := \frac{1}{2}\left(f_{i+\hf, j, k} + f_{i-\hf, j, k} \right),	 \quad d_x f_{i, j, k} := \frac{1}{h}\left(f_{i+\hf, j, k} - f_{i-\hf, j, k} \right).
	\end{eqnarray*}
Average and difference operators in $y$ and $z$ directions, denoted as $A_y$, $A_z$, $D_y$, $D_z$, $a_y$, $a_z$, $d_y$, and $d_z$, can be analogously defined. The discrete gradient and discrete divergence are given by
	\[
\nabh f_{i,j,k} =\left( D_xf_{i+\hf, j, k},  D_yf_{i, j+\hf, k},D_zf_{i, j, k+\hf}\right) , 
  \quad 
\nabh\cdot\vec{f}_{i,j,k} = d_x f^x_{i,j,k}	+ d_y f^y_{i,j,k} + d_z f^z_{i,j,k},
	\]
where $\vec{f} = (f^x,f^y,f^z)\in \vec{\mathcal{E}}_{\rm per}$. The standard discrete Laplacian becomes  
	\begin{align*}
\Delta_h f_{i,j,k} := & \nabla_{h}\cdot\left(\nabla_{h}f\right)_{i,j,k} =  d_x(D_x f)_{i,j,k} + d_y(D_y f)_{i,j,k}+d_z(D_z f)_{i,j,k}.
	\end{align*}
For a periodic scalar function $\mathcal{D}$ that is defined at face center points and $\vec{f}\in\vec{\mathcal{E}}_{\rm per}$, then $\mathcal{D}\vec{f}\in\vec{\mathcal{E}}_{\rm per}$, in the sense of point-wise multiplication, and we denote
	\[
\nabla_h\cdot \big(\mathcal{D} \vec{f} \big)_{i,j,k} = d_x\left(\mathcal{D}f^x\right)_{i,j,k}  + d_y\left(\mathcal{D}f^y\right)_{i,j,k} + d_z\left(\mathcal{D}f^z\right)_{i,j,k} .
	\]
If $f\in \mathcal{C}_{\rm per}$, then $\nabla_h \cdot\left(\mathcal{D} \nabla_h  \cdot \right):\mathcal{C}_{\rm per} \rightarrow \mathcal{C}_{\rm per}$ is defined as 
	\[
\nabla_h\cdot \big(\mathcal{D} \nabla_h f \big)_{i,j,k} = d_x\left(\mathcal{D}D_xf\right)_{i,j,k}  + d_y\left(\mathcal{D} D_yf\right)_{i,j,k} + d_z\left(\mathcal{D}D_zf\right)_{i,j,k} .
	\]
	
In addition, we introduce the following grid inner products
	\begin{equation*}
	\begin{aligned}
\ciptwo{f}{\xi}  &:= h^3\sum_{i,j,k=1}^N  f_{i,j,k}\, \xi_{i,j,k},\quad f,\, \xi\in {\mathcal C}_{\rm per},\quad
&\eipx{f}{\xi} := \langle a_x(f\xi) , 1 \rangle ,\quad f,\, \xi\in{\mathcal E}^{\rm x}_{\rm per},
\\
\eipy{f}{\xi} &:= \langle a_y(f\xi) , 1 \rangle ,\quad f,\, \xi\in{\mathcal E}^{\rm y}_{\rm per},\quad
&\eipz{f}{\xi} := \langle a_z(f\xi) , 1 \rangle ,\quad f,\, \xi\in{\mathcal E}^{\rm z}_{\rm per}.
\\
[ \vec{f}_1 , \vec{f}_2 ] &: = \eipx{f_1^x}{f_2^x}	+ \eipy{f_1^y}{f_2^y} + \eipz{f_1^z}{f_2^z}, \quad &\vec{f}_i = (f_i^x,f_i^y,f_i^z) \in \vec{\mathcal{E}}_{\rm per}, \ i = 1,2.
	\end{aligned}
	\end{equation*}	
Then, the following norms can be defined for grid functions. If $f\in {\mathcal C}_{\rm per}$, then $\nrm{f}_2^2 := \langle f , f \rangle$, $\nrm{f}_p^p := \ciptwo{|f|^p}{1}$, for $1\le p< \infty$, and $\nrm{f}_\infty := \max_{1\le i,j,k\le N}\left|f_{i,j,k}\right|$. The gradient norms are given by
	\begin{eqnarray*} 
\nrm{ \nabla_h f}_2^2 &: =& \eipvec{\nabh f }{ \nabh f } = \eipx{D_xf}{D_xf} + \eipy{D_yf}{D_yf} +\eipz{D_zf}{D_zf},  \quad \forall \,  f \in{\mathcal C}_{\rm per} ,   
	\\
\nrm{\nabla_h f}_p^p &:=&  \eipx{|D_xf|^p}{1} + \eipy{|D_yf|^p}{1} +\eipz{|D_zf|^p}{1}   , \quad \forall \, f \in{\mathcal C}_{\rm per}, \quad  1\le p<\infty . 
	\end{eqnarray*}
The higher-order norms can be similarly defined:
	\[
\nrm{f}_{H_h^1}^2 : =  \nrm{f}_2^2+ \nrm{ \nabla_h f}_2^2, \quad \nrm{f}_{H_h^2}^2 : =  \nrm{f}_{H_h^1}^2  + \nrm{ \Delta_h f}_2^2  , \quad \forall \, f \in{\mathcal C}_{\rm per}.
	\]

In addition, we recall the following inequalities that shall be used in the convergence analysis. 
\begin{lemma} \cite{guo16, guo2021} \label{Lem1}
For any $ f \in \mathring{\calC}_{per}$, we have 
\begin{align}
\| f \|_2 +  \| \nabla_h f \|_2 \le C \| \Delta_h f \|_2 , \label{lem1:1}
\end{align}
where $C$ is a positive constant independent of $h$. In addition, the following estimates are available: 
\begin{align} 
  & 
   \| f \|_\infty \le C ( \| f \|_2 + \| \Delta_h f \|_2 ) ,  \quad \forall f \in  {\calC}_{per} , 
   \label{lem1:2}   
\\
  & 
  \| \nabla_h f \|_\infty \le C h^{-1} \| f \|_\infty 
  \le C h^{-1} \| \Delta_h f \|_2  ,  \quad \forall f \in  \mathring{\calC}_{per} .  
   \label{lem1:3}   
\end{align}    
\end{lemma}


\subsection{Numerical schemes}
Based on the Slotboom reformulation~\eqref{NPe}, spatial discretization at $(x_i,y_j,z_k)$ leads to a semi-discrete scheme~\cite{LW14, DingWangZhou_JCP2019, LiuMaimaiti2021}
\begin{equation}\label{NPEs}
\begin{aligned}
\frac{d}{dt}c^l_{i,j,k}=&\nabla_h\cdot \big(e^{-S^l} \hat{g}^l \big)_{i,j,k},
\end{aligned}
\end{equation}
where  $e^{-S^l}  \hat{g}^l = (e^{-S^l} \hat{g}^l_x, e^{-S^l} \hat{g}^l_y, e^{-S^l} \hat{g}^l_z) \in \vec{\mathcal{E}}_{\rm per}$ with
\[
\hat{g}^l_{x,i+\frac{1}{2},j,k} = D_x \left(c^l e^{S^l} \right)_{i+\frac{1}{2},j,k}, 
\hat{g}^l_{y,i,j+\frac{1}{2},k} = D_y \left(c^l e^{S^l} \right)_{i,j+\frac{1}{2},k}, 
\mbox{~and ~} \hat{g}^l_{x,i,j,k+\frac{1}{2}} = D_z \left(c^l e^{S^l} \right)_{i,j,k+\frac{1}{2}}.
\]
In terms of numerical approximation of the mobility function $e^{-S^l}$ on the half-grid points, there are several options
\begin{equation}
\begin{aligned}
&\mbox{Harmonic mean: } 
&e^{-S^l_{i+{\frac{1}{2}},j,k}}&=\frac{2e^{-S^l_{i+1,j,k}}e^{-S^l_{i,j,k}}}{e^{-S^l_{i+1,j,k}}+e^{-S^l_{i,j,k}}} = \Big( \frac{e^{ S^l_{i+1,j,k}}+e^{ S^l_{i,j,k}}}{2}  \Big)^{-1} ,\\
&\mbox{Geometric mean: } 
&e^{-S^l_{i+{\frac{1}{2}},j,k}}&=e^{-\frac{S^l_{i+1,j,k}+S^l_{i,j,k}}{2}},\\
&\mbox{Arithmetic mean: } 
&e^{-S^l_{i+{\frac{1}{2}},j,k}}&=\frac{e^{-S^l_{i+1,j,k}}+e^{-S^l_{i,j,k}}}{2},\\
&\mbox{Entropic mean: }
&e^{-S^l_{i+{\frac{1}{2}},j,k}}&=\frac{S^l_{i+1,j,k}-S^l_{i,j,k}}{e^{S^l_{i+1,j,k}}-e^{S^l_{i,j,k}}}.\\
\end{aligned}
\label{mean-1} 
\end{equation}
The approximation of $e^{-S^l_{i,j+\frac{1}{2},k}}$ and $e^{-S^l_{i,j,k+\frac{1}{2}}}$ can be analogously defined. The periodic boundary conditions are imposed at a discrete level: 
\begin{equation}\label{PBC2}
c^l_{\frac{1}{2},j,k}=c^l_{N+\frac{1}{2},j,k},~~c^l_{i,\frac{1}{2},k}=c^l_{i,N+\frac{1}{2},k},~~c^l_{i,j,\frac{1}{2}}=c^l_{i,j,N+\frac{1}{2}}~~\mbox{for}~i,j,k=1,\cdots,N.\\
\end{equation}

\begin{remark}
The spatial discretization with the entropic mean of mobility functions on half-grid points can lead to the well-known Scharfetter--Gummel scheme, which is reduced to an upwind scheme when the convection is large. 
\end{remark}

Given the semi-discrete approximations $c^{l}_{i,j,k}$, the Poisson's equation is discretized as
\begin{equation}\label{DPssn}
-\kappa\Delta_h\psi_{i,j,k}=\sum_{l=1}^Mq^lc^{l}_{i,j,k}+\rho^f_{i,j,k}~~\mbox{for}~~ i,j,k=1, \dots, N.
\end{equation}
The periodic boundary conditions could be formulated as
\begin{equation}\label{PBC1}
\psi_{\frac{1}{2},j,k} =\psi_{N+\frac{1}{2},j,k},~\psi_{i,\frac{1}{2},k} =\psi_{i,N+\frac{1}{2},k},~\psi_{i,j,\frac{1}{2}}=\psi_{i,j,N+\frac{1}{2}}~~\mbox{for}~i,j,k=1,\cdots,N.
\end{equation}

\begin{remark}\label{re:unq}
To guarantee the existence of a solution to~\reff{DPssn}, it is assumed that the initial conditions and fixed charge distribution satisfy $\sum_{l=1}^Mq^lc^{l}_{\rm in}+\rho^f \in \mathring{\mathcal C}_{\rm per}$. In addition, we consider a zero-average constraint for $\psi$, i.e., $\psi \in \mathring{\mathcal C}_{\rm per}$, to make the solution to the discrete Poisson's equation unique.
\end{remark}

With a uniform time step size $\Delta t$ and $t^n= n \Delta t$, a semi-implicit scheme is proposed: Given $\psi^{n}\in \mathring{\mathcal{C}}_{per}$ and $c^{l,n}\in \mathcal{C}_{per},~l=1,\cdots,M$, find $c^{l,n+1}\in \mathcal{C}_{per},~l=1,\cdots,M$, $\psi^{n+1}\in \mathring{\mathcal{C}}_{per}$ by solving
\begin{equation}\label{FuPNP}
\left\{
\begin{aligned}
&\frac{c^{l,n+1}-c^{l,n}}{\Delta t}=\nabla_h \cdot\left(e^{-q^l\psi^n}\nabla_h(c^{l,n+1}e^{q^l\psi^n})\right) , \quad l=1,\cdots,M,\\
&-\kappa\Delta_h\psi^{n+1}=\sum_{l=1}^M q^l c^{l,n+1}+\rho^f , \quad 
\overline{\psi^{n+1}} = 0 . 
\end{aligned}
\right.
\end{equation}
We shall refer the fully discrete scheme~\reff{FuPNP} as the ``Slotboom scheme'' in the following sections.

\section{Theoretical properties} \label{sec:Properties}
In this section, we recall several important properties of the semi-implicit scheme~\reff{FuPNP} in preserving mass conservation, ionic positivity, and energy dissipation at a discrete level. The detailed proof could be found in the works~\cite{LW14, DingWangZhou_JCP2019, DingWangZhou_JCP2020, LiuMaimaiti2021}.

\begin{theorem}\label{t:Conservation}
{\bf (Mass conservation)} The semi-implicit scheme~\reff{FuPNP} conserves total ionic mass,  i.e.,
\begin{align}  
 h^3 \sum_{i,j,k=1}^{N}c_{i,j,k}^{l,n+1}  = h^3 \sum_{i,j,k=1}^{N}c_{i,j,k}^{l,n} . 
  \label{conserve-1}
\end{align}
\end{theorem}

\begin{theorem}\label{t:Postivity}
{\bf (Positivity preserving)} The semi-implicit scheme~\reff{FuPNP} is positivity-preserving in the sense that if $c_{i,j,k}^{l,n} > 0$, then
$$
c_{i,j,k}^{l,n+1}> 0 ~\mbox{ for  } i,j,k=1,\cdots, N.
$$
\end{theorem}


The fully discrete free energy is given by
 \begin{equation}\label{Fh}
 F^n_h =\sum_{l=1}^M \sum_{i,j,k=1}^{N}h^3\left[c^{l,n}_{i,j,k}\log c^{l,n}_{i,j,k}+\frac{1}{2}(q^lc^{l,n}_{i,j,k}+\rho^f_{i,j,k})\psi^n_{i,j,k}\right],
 \end{equation}
which is a second-order accurate approximation of $F$ in~\reff{F} in space. It has been proved that the semi-implicit scheme~\reff{FuPNP} respects free-energy dissipation at a discrete level conditionally~\cite{LiuMaimaiti2021}. 

\begin{theorem} \label{theorem:energy}
 {\bf (Energy dissipation)} If the time step size $\Delta t$ satisfies the sufficient condition $0 < \dt \le \tau^*$, where
\begin{equation} 
\tau^* = \frac{\kappa}{C_1}e^{-h|q^l|\| \nabla_h\psi^n\|_{\infty}} , \quad \mbox{with}~
C_1=\overset{M}{\underset{l=1}{\sum}}|q^l|^2\underset{1\leq l\leq M}{\max}\{\| c^{l,n+1}\|_{\infty}\}, 
  \label{bound-dt-1} 
\end{equation} 
then the discrete free energy $F_h^n$ is non-increasing, in the sense that
  \begin{equation}\label{dFh/dt}
 F_h^{n+1}-F_h^n\leq -\frac{\Delta t}{2}I^n,
 \end{equation}
 where
\[
 \begin{aligned}
I^n=&\sum_{l=1}^{M}\sum_{i,j,k=1}^Nh^3e^{-S^{l,n}_{i+\frac{1}{2},j,k}}D_x(c^{l,n+1}e^{S^{l,n}})_{i+\frac{1}{2},j,k}D_x(\log(c^{l,n+1}e^{S^{l,n}}))_{i+\frac{1}{2},j,k}\\
&+\sum_{l=1}^{M}\sum_{i,j,k=1}^Nh^3e^{-S^{l,n}_{i,j+\frac{1}{2},k}}D_y(c^{l,n+1}e^{S^{l,n}})_{i,j+\frac{1}{2},k}D_y(\log(c^{l,n+1}e^{S^{l,n}}))_{i,j+\frac{1}{2},k}\\
&+\sum_{l=1}^{M}\sum_{i,j,k=1}^Nh^3e^{-S^{l,n}_{i,j,k+\frac{1}{2}}}D_z(c^{l,n+1}e^{S^{l,n}})_{i,j,k+\frac{1}{2}}D_z(\log(c^{l,n+1}e^{S^{l,n}}))_{i,j,k+\frac{1}{2}}\\
\geq& 0.
\end{aligned}
\]
\end{theorem}
Following the proof given in the work~\cite{LiuMaimaiti2021}, we prove this theorem for the semi-implicit scheme~\reff{FuPNP} with the entropic mean in Appendix~\ref{Ap:A}.

\begin{remark} 
It is noticed that, the sufficient condition for the time step size~\eqref{bound-dt-1} relies on the $\| \cdot \|_\infty$ norm of the concentration variable, as well as the $W_h^{1, \infty}$ norm of the electrostatic potential variable, at both the previous and next time steps. Nonetheless, both quantities (for the numerical solutions) are not automatically available; these quantities have to be justified by the convergence estimate of the numerical solution and the corresponding bounds for the exact solution. The details of these estimates shall be presented in the next section. Afterward, a theoretical justification of the discrete energy dissipation is complete. 
\end{remark} 
%

\section{Convergence analysis}\label{sec:convergence} 

Let $(\psi_e,c^1_e,c^2_e, \cdots,c^M_e)$ be the exact solutions for the PNP system \reff{PNP}. Define $c^l_N:=\mathcal{P}_N c^l_e(\cdot,t)$, $\psi_N:=\mathcal{P}_N \psi_e(\cdot,t)$, $\rho^f_N:=\mathcal{P}_N \rho^f(\cdot)$, the Fourier projection of the exact solution into $\mathcal{B}^K$, which is the space of trigonometric polynomials of degree to and including $K$ (with $N=2K+1$). The following projection approximation is standard: if $(c^l_e,\psi_e) \in L^{\infty}(0,T;H^m_{per}(\Omega))$, for any $m\in\mathbb{N}$ with $0\leq k\leq m$,
\begin{equation}
\begin{aligned}
\|\psi_N-\psi_e \|_{L^{\infty}(0,T;H^k)}&\leq Ch^{m-k}\| \psi_e\|_{L^{\infty}(0,T;H^m)},\\
\|c^l_N-c^l_e \|_{L^{\infty}(0,T;H^k)}&\leq Ch^{m-k}\| c^l_e\|_{L^{\infty}(0,T;H^m)} , \quad l=1,2,\cdots,M.
\end{aligned}
  \label{projection-ets-1} 
\end{equation}

Denote $\psi^n_N=\psi_N(\cdot,t^n)$ and $c^{l,n}_N=c^l_N(\cdot,t^n)$, $l=1,\cdots,M$. We use the mass conservative projection for the initial data: $\psi^{0}=\mathcal{P}_h \psi_N(\cdot,t=0)$, $c^{l,0}=\mathcal{P}_h c^l_N(\cdot,t=0)$, $l=1,\cdots,M$: 
\[
\psi^{0}_{i,j,k}:= \psi_N(x_i,y_j,z_k,t=0),~~c^{l,0}_{i,j,k}:= c^l_N(x_i,y_j,z_k,t=0) ,   
 \quad  l=1,\cdots,M.
\]
The error grid function is defined as
\begin{equation}\label{ca:eq1}
e^{n}_{\psi}:=\mathcal{P}_h \psi^{n}_N-\psi^{n},~~e^{l,n}:=\mathcal{P}_h c^{l,n}_N-c^{l,n} ,  \quad l=1,\cdots,M, \, \, \,  n\in\mathbb{N}^*.
\end{equation}
As indicated above, one can verify that $\bar{e}^{n}_{\psi}=0$,  $\bar{e}^{l,n}=0$, for $l=1,\cdots,M$,  $n\in\mathbb{N}^*$. 

\begin{theorem}\label{Th:convergence}
Let $e^{n}_{\psi}$ and $e^{l,n}$ be the error grid functions defined in~\reff{ca:eq1}.  Then, under the linear refinement requirement $C_1h\leq \Delta t\leq C_2h$, the following convergence result is available as $\dt, h \to 0$: 
\begin{equation}\label{th:eq1}
\sum_{l=1}^M \big\|e^{l,n} \big\|_2 +\Big( \Delta t \sum_{l=1}^M \sum_{k=1}^{n} \big\|\nabla_h e^{l, k} \big\|^2_2 \Big)^{\frac{1}{2}}+\big\|e^{n}_{\psi} \big\|_{H^2_h}\leq C(\Delta t+h^2),
\end{equation}
where the constant $C>0$ is independent of $\Delta t$ and $h$.
\end{theorem}

\subsection{Higher-order consistency analysis}
It follows from the truncation error analysis that the numerical solution to the Slotboom scheme~\reff{FuPNP} approximates the projection solution $(\psi_N, c^1_N,\cdots,c^M_N )$ with only first order accuracy in time and second order accuracy in space. Such a first order accuracy in time is not sufficient to recover an a-priori $\ell^{\infty}$ bound to ensure the energy stability analysis. Instead, a higher order consistency analysis is performed by adding perturbation terms to recover such a bound.  We first construct auxiliary variables, $c^l_{\Delta t}$, $\psi_{\Delta t}$, $\check{c}^l$, $\check{\psi}$, and $\check{\rho}^f$: 
\begin{equation} \label{s1:eq0}
\check{\psi} =\psi_N+ \dt \mathcal{P}_N  \psi_{\Delta t} , \, \, \check{\rho}^f=\rho^f_N, \, 
\check{c}^l =c^l_N+ \dt \mathcal{P}_N c^l_{\Delta t} , \quad l=1,\cdots,M, 
\end{equation}
so that the higher $O(\Delta t^2+h^2)$ consistency is achieved between the numerical solution and the constructed solution $(\check{\psi}, \check{c}^l)$. The constructed variables $\psi_{\Delta t}$, $c^l_{\Delta t}$, $l=1, \cdots,M$, which solely depend on the exact solutions $c^l_e$ and $\psi_e$, can be found by a perturbation expansion.


It follows from the Taylor expansion for the temporal discretization and the estimate for the projection solution that
\begin{equation}\label{s1:eq1}
\begin{aligned}
\frac{c^{l,n+1}_N-c^{l,n}_N}{\Delta t}=&\nabla\cdot\left(e^{-q^l\psi^{n}_N}\nabla(c^{l,n+1}_Ne^{q^l\psi^{n}_N})\right) +\Delta t G^{l,n}_{(0)}+O(\Delta t^2) , \quad l=1,\cdots,M, \\
-\kappa\Delta\psi^{n+1}_N=&\sum_{l=1}^M q^lc^{l,n+1}_N+\rho^f_N ,  \quad 
\int_\Omega \, \psi^{n+1}_N d {\bf x} = 0, 
\end{aligned}
\end{equation}
where the functions $G^{l,n}_{(0)}$, $l=1,\cdots,M,$ are smooth enough in the sense that the temporal and spatial derivatives are bounded. In addition, $G^{l,n}_{(0)}$ is a spatial function with sufficient regularity for a fixed time step $t^n$.

The temporal perturbation variables $\left(\psi_{\Delta t},c^1_{\Delta t},\cdots, c^M_{\Delta t}\right)$ can be found by solving 
\begin{equation*}
\begin{aligned}
&\partial_t c^l_{\Delta t}=\nabla\cdot\left(e^{-q^l\psi_N}\nabla(q^lc^l_Ne^{q^l\psi_N}\psi_{\Delta t})+e^{-q^l\psi_N} \nabla (e^{q^l\psi_N}c^l_{\Delta t}) -q^le^{-q^l\psi_N} \psi_{\Delta t}\nabla(c^l_Ne^{q^l\psi_N})\right)-G^l_{(0)} ,  \\
&-\kappa\Delta\psi_{\Delta t}=\sum_{l=1}^M q^l c^l_{\Delta t} ,  \quad 
\int_\Omega \, \psi_{\dt} d {\bf x} = 0 . 
\end{aligned}
\end{equation*}
This is a linear PDE system that depends on the projection solutions $(\psi_N,c^1_N,\cdots,c^M_N)$, and the existence of its solution is straightforward.  In addition, the derivatives of $(\psi_{\Delta t},c^1_{\Delta t},\cdots,c^M_{\Delta t})$ in various orders are bounded. A semi-implicit discretization leads to
\begin{equation}\label{s1:eq2}
\begin{aligned}
\frac{c^{l,n+1}_{\Delta t}-c^{l,n}_{\Delta t}}{\Delta t}=&\nabla\cdot\left(e^{-q^l\psi_N^n}\nabla(q^lc^{l,n+1}_Ne^{q^l\psi_N^n}\psi_{\Delta t}^n)+e^{-q^l\psi_N^n} \nabla (e^{q^l\psi_N^n}c^{l,n+1}_{\Delta t})\right.\\
&\qquad\left. -q^le^{-q^l\psi_N^n} \psi_{\Delta t}^n\nabla(c^{l,n+1}_Ne^{q^l\psi_N^n})\right)-G^{l,n}_{(0)}+O(\Delta t) , \quad l=1,\cdots,M,\\
-\kappa\Delta\psi^{n+1}_{\Delta t}=&\sum_{l=1}^M q^lc^{l,n+1}_{\Delta t} ,   \quad 
\int_\Omega \, \psi_{\dt}^{n+1} d {\bf x} = 0 . 
\end{aligned}
\end{equation}
Therefore, a combination of \reff{s1:eq1} and \reff{s1:eq2} gives the temporal truncation error for $\check{c}^l$ and $\check{\psi}$:
\begin{equation}\label{s1:eq3}
\begin{aligned}
\frac{\check{c}^{l,n+1}-\check{c}^{l,n}}{\Delta t}=&\nabla\cdot\left(e^{-q^l\check{\psi}^{n}}\nabla(\check{c}^{l,n+1} e^{q^l\check{\psi}^{n}})\right)+O(\Delta t^2) , \quad l=1,\cdots,M,\\
-\kappa\Delta\check{\psi}^{n+1}=&\sum_{l=1}^M q^l\check{c}^{l,n+1}+\check{\rho}^f ,  \quad 
\int_\Omega \, \check{\psi}^{n+1} d {\bf x} = 0 . 
\end{aligned}
\end{equation}
In fact, the following linearized expansions have been used in the derivation of \reff{s1:eq3}:
\[
\begin{aligned}
&e^{q^l\check{\psi}}=e^{q^l(\psi_N+\Delta t \mathcal{P}_N \psi_{\Delta t})}=e^{q^l\psi_N}+\Delta tq^l\mathcal{P}_N \psi_{\Delta t}e^{q^l\psi_N}+O(\Delta t^2),\\
&e^{-q^l\check{\psi}}=e^{-q^l(\psi_N+\Delta t \mathcal{P}_N \psi_{\Delta t})}=e^{-q^l\psi_N}-\Delta tq^l\mathcal{P}_N \psi_{\Delta t}e^{-q^l\psi_N}+O(\Delta t^2).
\end{aligned}
\]
After spatial discretization, one can obtain by the Taylor expansion for $(\check{\psi},\check{c}^1,\cdots,\check{c}^M)$:  
\begin{equation}\label{s1:eq7}
\begin{aligned}
\frac{\check{c}^{l,n+1}-\check{c}^{l,n}}{\Delta t}=&\nabla_h\cdot\left(e^{-q^l\check{\psi}^{n}}\nabla_h(\check{c}^{l,n+1}e^{q^l\check{\psi}^{n}})\right)+\tau^{l,n+1} , \quad l=1,\cdots,M,\\
-\kappa\Delta_h\check{\psi}^{n+1}=&\sum_{l=1}^M q^l\check{c}^{l,n+1}+\check{\rho}^f,  \quad 
 \overline{\check{\psi}^{n+1}} = 0 , 
\end{aligned}
\end{equation}
where 
$$\|\tau^{l,n+1} \|_2\leq C(\Delta t^2+h^2).$$

\begin{remark}
Since the correction functions only depend on $(\psi_N,c^1_N,c^2_N, \cdots,c^M_N)$ and the exact solution, we can obtain discrete $W^{1,\infty}$ bounds from the regularity of the constructed solutions:
\begin{equation}
\begin{aligned}\label{r:eq1}
&\| \check{c}^{l,n}\|_{\infty}\leq C_0^*,~&\| \nabla_h\check{c}^{l,n}\|_{\infty}&\leq C_0^* , \quad l=1,\cdots,M,\\
&\| \check{\psi}^{n}\|_{\infty}\leq C_0^*,~&\| \nabla_h\check{\psi}^{n}\|_{\infty}&\leq C_0^* , 
\end{aligned}
\end{equation}
at any time step $t^n$.
\end{remark}

\subsection{Error estimate}
Instead of working on the original numerical error functions defined in~\eqref{ca:eq1}, we first consider the following ones 
\begin{equation} 
\td{\psi}^n:=\mathcal{P}_h \check{\psi}^n-\psi^n,~~\td{c}^{l,n}:=\mathcal{P}_h \check{c}^{l,n}-c^{l,n} , \quad l=1,2,\cdots,M, \, \, N\in \mathbb{N}^*. \label{error function-1} 
\end{equation} 
By the consistency estimate~\reff{s1:eq7}, a higher-order truncation accuracy is available for these numerical error functions.

Subtracting the numerical scheme \reff{FuPNP} from the consistency estimate \reff{s1:eq7} yields
\begin{equation}
\begin{aligned}\label{s1:eq8}
\frac{\td{c}^{l,n+1}-\td{c}^{l,n}}{\Delta t}=&\nabla_h\cdot\left((e^{-q^l\check{\psi}^n}-e^{-q^l\psi^n})\nabla_h(\check{c}^{l,n+1} e^{q^l\check{\psi}^n})\right)+\nabla_h\cdot\left(e^{-q^l\psi^n}\nabla_h(\td{c}^{l,n+1} e^{q^l\psi^n})\right)\\
&+\nabla_h\cdot\left(e^{-q^l\psi^n}\nabla_h(\check{c}^{l,n+1}(e^{q^l\check{\psi}^n}-e^{q^l\psi^n}))\right)+\tau^{l,n+1} , \quad  l=1,\cdots,M,\\
-\kappa\Delta_h\td{\psi}^{n+1}=&\sum_{l=1}^Mq^l\td{c}^{l,n+1}, \quad 
\overline{\td{\psi}^{n+1}} = 0 , 
\end{aligned}
\end{equation}
where $\td{\psi}^{n+1},~ \td{c}^{l,n+1}\in \mathring{\calC}_{per}$,~$l=1,2,\cdots,M$. As an example, we now take the harmonic mean for the evaluation of the nonlinear mobility functions $e^{-q^l \psi^n}$ (or $e^{-q^l \check{\psi}^n}$) over the staggered mesh points. 

Since $\check{c}^{l,n+1}$ and $\check{c}^{l,n+1} e^{q^l\check{\psi}^n}$ only depend on the exact solutions and the constructed variables, one can assume a discrete $W^{1,\infty}$ bound 
	\begin{equation} 
\|  \nabla_h ( \check{c}^{l,n+1} e^{q^l\check{\psi}^n} ) \|_\infty ,  \, \, 
 \| \check{c}^{l,n+1} \|_\infty  + \| \nabla_h \check{c}^{l,n+1} \|_\infty \le C_0^\star . 
	\label{assumption:W1-infty bound-2}  
	\end{equation} 
In addition, the following a-priori assumption is made, so that the nonlinear analysis could be accomplished by an induction argument: 
\begin{eqnarray} 
  \| \tilde{c}^{l,n} \|_2  \le \dt^\frac{15}{8} + h^\frac{15}{8}  ,  \quad 
  l = 1, 2, \cdots, M  .  
   \label{a priori-1} 
\end{eqnarray} 
This a-priori assumption will be recovered by the optimal rate convergence analysis at the next time step, as will be demonstrated later. With the error equation for $\td{\psi}^{n+1}$ in~\eqref{s1:eq8}, we apply the preliminary inequalities~\eqref{lem1:2}, \eqref{lem1:3} in Lemma~\ref{Lem1} and obtain
\begin{align} 
  & 
  \| \td{\psi}^n \|_\infty   \le C \| \Delta_h \td{\psi}^n \|_2 
  \le C  \sum_{l=1}^M q^l \| \td{c}^{l,n} \|_2 \le C  ( \dt^\frac{15}{8} + h^\frac{15}{8} ) 
   \le \frac14 , \label{a priori-6-1}   
\\
  & 
   \| \nabla_h \tilde{\psi}^n \|_\infty  \le \frac{C \| \tilde{\psi}^n \|_\infty}{h} \le   \frac{C ( \dt^\frac{15}{8} + h^\frac{15}{8} ) }{h} \le  C ( \dt^\frac{7}{8} + h^\frac{7}{8} ) \le \frac14 ,  \label{a priori-6-2}    
\end{align}   
in which the fact that $\td{\psi}^n \in \mathring{\calC}_{per}$ has been applied. Therefore, its combination with the regularity assumption~\eqref{r:eq1} gives a $W_h^{1,\infty}$ bound for the numerical solution $\psi^n$ at the previous time step: 
	\begin{eqnarray} 
 \| \psi^n \|_\infty \le \| \check{\psi}^n  \|_\infty +  \| \td{\psi}^n \|_\infty 
  \le \tilde{C}_3:= C_0^\star + \frac14  ,  \quad 
 \| \nabla_h \psi^n \|_\infty \le \| \nabla_h \check{\psi}^n  \|_\infty +  \| \nabla_h \td{\psi}^n \|_\infty 
  \le \tilde{C}_3   .  \label{a priori-7} 
	\end{eqnarray}

Before proceeding with the convergence analysis, the following preliminary results are needed. 

\begin{proposition} \label{prop: prelim est} 
Under the a-priori assumption~\eqref{a priori-1} for the numerical error function at the previous time step, we have 
\begin{align} 
  & 
  \| e^{-q^l \psi^n} \|_\infty , \, \, \| \nabla_h ( e^{q^l \psi^n} ) \|_\infty  \le \tilde{C}_4 ,  
  \label{prelim est-0-1} 
\\
  & 
  \|  e^{-q^l \check{\psi}^n} - e^{-q^l \psi^n} \|_2 \le \tilde{C}_5 \| \td{\psi}^n \|_2 , 
  \label{prelim est-0-2}        
\\
  & 
  \|  e^{q^l \check{\psi}^n} - e^{q^l \psi^n} \|_2 \le \tilde{C}_6 \| \td{\psi}^n \|_2 ,  \quad 
   \|  \nabla_h ( e^{q^l \check{\psi}^n} - e^{q^l \psi^n} ) \|_2 
    \le \tilde{C}_7 ( \| \td{\psi}^n \|_2 + \| \nabla_h \td{\psi}^n \|_2 ) ,  \label{prelim est-0-3} 
\end{align} 
for any $l = 1, \cdots, M$, in which the harmonic mean formula in~\eqref{mean-1} has been applied in the evaluation of $e^{-q^l \check{\psi}^n}$ and $e^{-q^l \psi^n}$ at the staggered grid points. Notice that the constants $\tilde{C}_j$ ($4 \le j \le 7$) are independent of $\dt$ and $h$. 
\end{proposition}

\begin{proof} 
By the harmonic mean formula  for $e^{-q^l \psi^n}$ (in~\eqref{mean-1}), we find that 
\begin{equation} 
  \Big| e^{-q^l \psi^n_{i+{\frac{1}{2}},j,k}} \Big| 
  = \Big( \frac{e^{ q^l \psi^n_{i+1,j,k}}+ e^{ q^l \psi^n_{i,j,k}}}{2}  \Big)^{-1}   
  \le \Big( e^{ |q^l |\cdot (-\tilde{C}_3) } \Big)^{-1}   
  =  e^{ |q_l| \tilde{C}_3 } ,  \label{prelim est-1}   
\end{equation} 
at each staggered grid point $(i+\frac12, j,k)$, in which the $\| \cdot \|_\infty$ bound~\eqref{a priori-7} has been used. Similar derivations could be applied at the staggered grid points in the $y$ and $z$ directions. Then we get $\| e^{-q^l \psi^n} \|_\infty \le e^{ |q^l| \tilde{C}_3 }$.  Using similar argument, we are able to prove that $\| e^{-q^l \check{\psi}^n} \|_\infty \le e^{ |q^l |C_0^\star }$. 

For the second inequality in~\eqref{prelim est-0-1}, we observe the following expansion at each numerical mesh cell, from $(i,j,k)$ to $(i+1,j,k)$: 
\begin{equation} 
  D_x ( e^{q^l \psi^n} )_{i+{\frac{1}{2}},j,k}  
  = \frac{e^{ q^l \psi^n_{i+1,j,k}} - e^{ q^l \psi^n_{i,j,k}}}{h}   
  = q^l e^{ q^l \xi^{(1)}} ( D_x \psi^n )_{i+{\frac{1}{2}},j,k}  ,  \label{prelim est-2-1}   
\end{equation} 
in which the intermediate value theorem has been applied, and $\xi^{(1)}$ is between $\psi^n_{i,j,k}$  and $\psi^n_{i+1,j,k}$. Then we conclude that 
\begin{equation} 
   \|  e^{ q^l \xi^{(1)}} \|_\infty \le e^{|q^l| \cdot \| \psi^n \|_\infty} 
   \le e^{|q^l| \tilde{C}_3} .  \label{prelim est-2-2}   
\end{equation} 
As a consequence, an application of discrete H\"older inequality gives 
\begin{equation} 
  \| D_x ( e^{q^l \psi^n} ) \|_\infty \le  \| e^{ q^l \xi^{(1)}} \|_\infty 
  \cdot \|  D_x \psi^n \|_\infty \le |q^l | e^{|q^l | \tilde{C}_3}  \tilde{C}_3 ,  \label{prelim est-2-3}   
\end{equation}  
in which the a-priori estimate~\eqref{a priori-7} has been applied in the last step. Similar bounds could be derived for $\| D_y ( e^{q^l \psi^n} ) \|_\infty$, $\| D_z ( e^{q^l \psi^n} ) \|_\infty$, respectively. Then we obtain $\| \nabla_h ( e^{q^l \psi^n} ) \|_\infty \le |q^l | e^{|q^l | \tilde{C}_3}  \tilde{C}_3$.  As a result, both inequalities in~\eqref{prelim est-0-1} have been proved, by taking $\tilde{C}_4 = \max ( e^{ |q^l| \tilde{C}_3 }, |q^l | e^{|q^l | \tilde{C}_3}  \tilde{C}_3 )$. 
  
  To prove the inequality~\eqref{prelim est-0-2}, we see that the expansion~\eqref{prelim est-1} indicates the following identity: 
\begin{equation} 
\begin{aligned} 
  & 
    e^{-q^l \check{\psi}^n_{i+{\frac{1}{2}},j,k}} - e^{-q^l \psi^n_{i+{\frac{1}{2}},j,k}}  
\\
  = & 
    - e^{-q^l \check{\psi}^n_{i+{\frac{1}{2}},j,k}} e^{-q^l \psi^n_{i+{\frac{1}{2}},j,k}}     
  \Big( \frac{e^{ q^l \check{\psi}^n_{i+1,j,k}} + e^{ q^l \check{\psi}^n_{i,j,k}}}{2}  
  - \frac{e^{ q^l \psi^n_{i+1,j,k}}+ e^{ q^l \psi^n_{i,j,k}}}{2}  \Big)
\\
  = & 
    - e^{-q^l \check{\psi}^n_{i+{\frac{1}{2}},j,k}} e^{-q^l \psi^n_{i+{\frac{1}{2}},j,k}}    
  \Big( \frac{e^{ q^l \check{\psi}^n_{i+1,j,k}} - e^{ q^l \psi^n_{i+1,j,k}} }{2}  
  + \frac{ e^{ q^l \check{\psi}^n_{i,j,k}} - e^{ q^l \psi^n_{i,j,k}}}{2}  \Big)     
\\
  = & 
    - \frac{q^l}{2} e^{-q^l \check{\psi}^n_{i+{\frac{1}{2}},j,k}} e^{-q^l \psi^n_{i+{\frac{1}{2}},j,k}}    
  \Big( e^{ q^l \xi^{(2)} } \td{\psi}^n_{i+1,j,k}   + e^{ q^l \xi^{(3)} } \td{\psi}^n_{i,j,k}  \Big)  ,   
\end{aligned} 
  \label{prelim est-3-1}   
\end{equation} 
in which the intermediate value theorem has been applied, $\xi^{(2)}$ is between $\psi^n_{i+1,j,k}$  and $\check{\psi}^n_{i+1,j,k}$, $\xi^{(3)}$ is between $\psi^n_{i,j,k}$  and $\check{\psi}^n_{i,j,k}$, respectively. By the regularity assumption~\eqref{r:eq1} and the a-priori $\| \cdot \|_\infty$ bound~\eqref{a priori-7}, we see that 
\begin{equation} 
   \| e^{ q^l \xi^{(2)} } \|_\infty , \, \,  \| e^{ q^l \xi^{(3)} } \|_\infty 
   \le e^{|q^l| \tilde{C}_3} .    \label{prelim est-3-2}   
\end{equation}    
Therefore, an application of discrete H\"older inequality reveals that 
\begin{equation} 
\begin{aligned} 
  & 
    \| e^{-q^l \check{\psi}^n } - e^{-q^l \psi^n }  \|_2 
\\
  \le & 
    \frac{|q^l|}{2}  \| e^{-q^l \check{\psi}^n } \|_\infty \cdot \|  e^{-q^l \psi^n }   \|_\infty    
  \Big( \| e^{ q^l \xi^{(2)} } \|_\infty \cdot \| \td{\psi}^n \|_2   
  + \| e^{ q^l \xi^{(3)} } \|_\infty \cdot \| \td{\psi}^n \|_2  \Big) 
\\
  \le & 
    \frac{|q^l|}{2}    e^{ |q^l| \tilde{C}_3 } \cdot  e^{ |q^l| \tilde{C}_0^* }    
  \cdot 2 e^{|q^l| \tilde{C}_3} \| \td{\psi}^n \|_2  \le |q^l | e^{3 q^l \tilde{C}_3} \| \td{\psi}^n \|_2  ,  
\end{aligned} 
  \label{prelim est-3-3}   
\end{equation} 
in which~\eqref{prelim est-0-1} is recalled. This proves the inequality~\eqref{prelim est-0-2}, by taking $\tilde{C}_5 = |q^l | e^{3 q^l \tilde{C}_3}$.      

To prove the inequalities~\eqref{prelim est-0-3}, we begin with the following difference formula at each grid point: 
\begin{equation} 
    e^{q^l \check{\psi}^n_{i,j,k}} - e^{q^l \psi^n_{i,j,k}} 
     = q^l e^{ q^l \xi^{(4)}_{i,j,k}}  \td{\psi}^n_{i,j,k}  ,  \label{prelim est-4-1}   
\end{equation} 
using the intermediate value theorem, with $\xi^{(4)}_{i,j,k}$ being between $\psi^n_{i,j,k}$  and $\check{\psi}^n_{i,j,k}$. Therefore, we have
\begin{equation} 
  \| \xi^{(4)}  \|_\infty  \le \max (  \| \psi^n \|_\infty ,  \| \check{\psi}^n \|_\infty ) 
  \le \max ( \tilde{C}_3 , C_0^* ) = \tilde{C}_3 , \quad 
  \| e^{ q^l  \xi^{(4)} } \|_\infty  \le e^{|q^l | \cdot\|  \xi^{(4)} \|_\infty} \le   e^{|q^l|  \tilde{C}_3 } . 
  \label{prelim est-4-2}   
\end{equation}   
In turn, an application of discrete H\"older inequality gives 
\begin{equation} 
    \| e^{q^l \check{\psi}^n } - e^{q^l \psi^n } \|_2 
    \le  \| e^{ q_l  \xi^{(4)} } \|_\infty \cdot   \| \td{\psi}^n \|_2  
    \le  |q^l|  e^{|q^l|   \tilde{C}_3 }   \| \td{\psi}^n \|_2  ,  \label{prelim est-4-3}   
\end{equation}  
so that the first inequality in~\eqref{prelim est-0-3} is proved, by taking $\tilde{C}_6 =  |q^l|  e^{|q^l|   \tilde{C}_3 }$. 

To prove the second inequality in~\eqref{prelim est-0-3}, we recall the point-wise expansion~\eqref{prelim est-4-1}, and look at the numerical mesh cell from $(i,j,k)$ to $(i+1,j,k)$: 
\begin{equation} 
    e^{q^l \check{\psi}^n_{i,j,k}} - e^{q^l \psi^n_{i,j,k}} 
     = q^l e^{ q^l \xi^{(4)}_{i,j,k}}  \td{\psi}^n_{i,j,k}  ,  \quad 
    e^{q^l \check{\psi}^n_{i+1,j,k}} - e^{q^l \psi^n_{i+1,j,k}} 
     = q^l e^{ q^l \xi^{(4)}_{i+1,j,k}}  \td{\psi}^n_{i+1,j,k}  , \label{prelim est-5-1}   
\end{equation} 
which in turn leads to the gradient expansion in the $x$ direction:  
\begin{equation}
  D_x  ( e^{q^l \check{\psi}^n } - e^{q^l \psi^n } )_{i+\frac12, j,k} 
  = q^l e^{ q^l \xi^{(4)}_{i,j,k}}  ( D_x \td{\psi} )_{i+\frac12, j,k}  
     + \frac{ e^{q^l \xi^{(4)}_{i+1,j,k}}   - e^{ q^l \xi^{(4)}_{i,j,k}}  }{h} q^l \td{\psi}^n_{i+1,j,k} .   
\label{prelim est-5-2} 
\end{equation}   
With the $\| \cdot \|_\infty$ bound in~\eqref{prelim est-4-2}, it is straightforward to get an $\| \cdot \|_2$ estimate for the first term. For the second term, we begin with the following observation: 
\begin{equation} 
   e^{ q^l \xi^{(4)}_{i,j,k}} -  e^{ q^l \psi^n_{i,j,k}}  
   =  q^l  e^{ q^l \xi^{(5)}_{i,j,k}}  (  \xi^{(4)}_{i,j,k} - \psi^n_{i,j,k} ) , \quad 
   \mbox{with $\xi^{(5)}_{i,j,k}$ between $\xi^{(4)}_{i,j,k}$ and $\psi^n_{i,j,k}$ } . 
   \label{prelim est-5-3} 
\end{equation}  
Moreover, by the fact that $\xi^{(4)}_{i,j,k}$ is between $\psi^n_{i,j,k}$ and $\check{\psi}^n_{i,j,k}$, we conclude that 
\begin{align} 
  & 
     \| \xi^{(5)}  \|_\infty  \le \max (  \| \psi^n \|_\infty ,  \| \check{\psi}^n \|_\infty ) 
  \le \max ( \tilde{C}_3 , C_0^* ) = \tilde{C}_3 , \quad 
     | e^{ q^l \xi^{(5)}_{i,j,k}} |  \le e^{|q^l \cdot |  \| \xi^{(5)}  \|_\infty } 
     \le e^{|q^l | \tilde{C}_3} ,   \label{prelim est-5-4-1}      
\\
  & 
  | \xi^{(4)}_{i,j,k} - \psi^n_{i,j,k}  | \le | \check{\psi}^n_{i,j,k} - \psi^n_{i,j,k}  |     
  = | \td{\psi}^n_{i,j,k}  |   ,   \label{prelim est-5-4-2}  
\\
  & 
  \mbox{so that} \quad 
  | e^{ q^l \xi^{(4)}_{i,j,k}} -  e^{ q^l \psi^n_{i,j,k}}  | 
  \le C  e^{|q^l| \tilde{C}_3}   ( \dt^\frac{15}{8} + h^\frac{15}{8} ) , \label{prelim est-5-4-3}   
\end{align}      
in which the a-priori $\| \cdot \|_\infty$ estimate~\eqref{a priori-6-1} has been applied. Using similar arguments, the following estimate is also available: 
\begin{equation} 
  | e^{ q^l \xi^{(4)}_{i+1,j,k}} -  e^{ q^l \psi^n_{i+1,j,k}}  | 
  \le C  e^{|q^l | \tilde{C}_3}   ( \dt^\frac{15}{8} + h^\frac{15}{8} ) . \label{prelim est-5-4-4}   
\end{equation}  
Then we arrive at 
\begin{equation} 
\begin{aligned} 
  \Big|  \frac{ e^{q^l \xi^{(4)}_{i+1,j,k}}   - e^{ q^l \xi^{(4)}_{i,j,k}}  }{h}   \Big|  
  \le  &   \Big|  \frac{ e^{q^l \psi^n_{i+1,j,k}}   - e^{ q^l \psi^n_{i,j,k}}  }{h}   \Big|   
   + \frac{C  e^{|q^l| \tilde{C}_3}   ( \dt^\frac{15}{8} + h^\frac{15}{8} )}{h}  
\\
  \le & 
  \| \nabla_h  (e^{q^l \psi^n }) \|_\infty + \frac14 \le \tilde{C}_4 + \frac14, 
\end{aligned}   
   \label{prelim est-5-5} 
\end{equation} 
under the linear refinement constraint $C_1 h \le \dt \le C_2 h$, and the bound~\eqref{prelim est-0-1} has been recalled in the last step. Since this inequality is valid at any numerical mesh, we get 
\begin{equation} 
  \| D_x ( e^{q^l \xi^{(4)} } ) \|_\infty \le \tilde{C}_4 + \frac14 .   
   \label{prelim est-5-6} 
\end{equation}   
Therefore, an application of discrete H\"older inequality to the expansion formula~\eqref{prelim est-5-2} indicates that 
\begin{equation}
\begin{aligned} 
  \| D_x  ( e^{q^l \cdot \check{\psi}^n } - e^{q^l \psi^n } ) \|_2 
   \le  & |q^l |\cdot \| e^{ q^l \xi^{(4)} } \|_\infty \cdot \| D_x \td{\psi} \|_2  
     + |q^l | \cdot \| D_x ( e^{q^l \xi^{(4)} } ) \|_\infty \cdot \|  \td{\psi}^n \|_2  
\\
  \le & 
      |q^l |\cdot   e^{|q^l| \tilde{C}_3}  \| D_x \td{\psi} \|_2  
     + |q^l | ( \tilde{C}_4 + \frac14 )  \|  \td{\psi}^n \|_2 .       
\end{aligned} 
\label{prelim est-5-7} 
\end{equation}   
Similar estimates could be derived for the discrete gradient in the $y$ and $z$ directions, and we are able to obtain the following inequality: 
\begin{equation} 
  \| \nabla_h  ( e^{q^l \check{\psi}^n } - e^{q^l \psi^n } ) \|_2 
  \le  |q^l | e^{|q^l | \tilde{C}_3}  \| \nabla_h \td{\psi} \|_2  
     + \sqrt{3}  |q^l |  ( \tilde{C}_4 + \frac14 )  \|  \td{\psi}^n \|_2 .       
\label{prelim est-5-8} 
\end{equation}   
As a result, the second inequality of~\eqref{prelim est-0-3} is established, by taking $\tilde{C}_7 =  |q^l |\max( e^{|q^l| \tilde{C}_3} , \sqrt{3} ( \tilde{C}_4 + \frac14 ) )$. This finishes the proof of Proposition~\ref{prop: prelim est}.  \qed 
\end{proof}

Now we proceed with the error estimate. Taking a discrete inner product with by $2\td{c}^{l,n+1}$ leads to
\begin{equation} \label{convergence-1}
\begin{aligned} 
  & 
\frac{1}{\Delta t}\left(  \| \td{c}^{l,n+1} \|^2_2 - \|\td{c}^{l,n} \|^2_2 + \|\td{c}^{l,n+1}-\td{c}^{l,n} \|^2_2 \right) + 2 \Big\langle e^{-q^l\psi^n} \nabla_h (\td{c}^{l,n+1}e^{q^l\psi^n}),  \nabla_h\td{c}^{l,n+1} \Big\rangle \\
=&- 2 \Big\langle (e^{-q^l\check{\psi}^n}-e^{-q^l\psi^n})\nabla_h(\check{c}^{l,n+1}e^{q^l\check{\psi}^n}),  \nabla_h\td{c}^{l,n+1} \Big\rangle \\
&- 2 \Big\langle e^{-q^l \psi^n} \nabla_h (\check{c}^{l,n+1} (e^{q^l\check{\psi}^n} - e^{q^l\psi^n} ) ),  \nabla_h \td{c}^{l,n+1} \Big\rangle + 2 \langle \tau^{l,n+1}, \td{c}^{l,n+1} \rangle .
\end{aligned}
\end{equation}
 The bound for the local truncation error term is straightforward: 
\begin{equation} \label{convergence-2}
2 \langle \tau^{l,n+1},\td{c}^{l,n+1} \rangle \leq \|\tau^{l,n+1} \|_2^2 + \|\td{c}^{l,n+1} \|_2^2.
\end{equation}
For the nonlinear term on the left hand side, we examine a single numerical mesh cell, from $(i,j,k)$ to $(i+1,j,k)$, and observe the following expansion identity: 
\begin{equation} 
  D_x ( f g )_{i, j, k} =  (A_x f)_{i,j,k}  ( D_x g )_{i,j,k}  
  +  (A_x g)_{i,j,k}  ( D_x f )_{i,j,k}  ,  \, \, \, 
  (A_x f)_{i,j,k}  =  \frac12 ( f_{i,j,k}   +  f_{i+1,j,k}  ) . 
  \label{expansion-1} 
\end{equation} 
Thus, we get 
\begin{align} 
  & 
  D_x (\td{c}^{l,n+1} e^{q^l\psi^n}) 
  = A_x \td{c}^{l,n+1} \cdot  D_x e^{q^l\psi^n}   
  +  A_x e^{q^l \psi^n}  \cdot D_x  \td{c}^{l,n+1} ,  \label{convergence-3-1} 
\\
  & 
  e^{-q^l \psi^n} D_x (\td{c}^{l,n+1} e^{q^l \psi^n})  
  = A_x \td{c}^{l,n+1} \cdot  D_x e^{q^l \psi^n}  \cdot e^{-q^l \psi^n}  
  +   D_x  \td{c}^{l,n+1}  ,  \label{convergence-3-2} 
\end{align}   
in which the following point-wise identity has been applied in the derivation: 
\begin{equation} 
   e^{-q^l \psi^n_{i+\frac12, j, k} }  \cdot ( A_x e^{q^l \psi^n} )_{i+\frac12, j,k} 
   =  \Big( \frac{e^{ q^l \psi^n_{i+1,j,k}}+e^{ q^l \psi^n_{i,j,k}}}{2}  \Big)^{-1} 
   \cdot  \frac{e^{ q^l \psi^n_{i+1,j,k}}+e^{ q^l \psi^n_{i,j,k}}}{2} = 1 , \, \, \, 
   \mbox{(by~\eqref{mean-1})} .  \label{convergence-3-3} 
\end{equation}   
Then we arrive at 
\begin{equation} 
\begin{aligned} 
  & 
   \Big\langle e^{-q^l\psi^n} D_x (\td{c}^{l,n+1} e^{q^l\psi^n}),  D_x \td{c}^{l,n+1} \Big\rangle   
\\
  = & 
    \|  D_x \td{c}^{l,n+1} \|_2^2 
   +  \Big\langle A_x \td{c}^{l,n+1}   ( D_x e^{q^l \psi^n}  ) e^{-q^l \psi^n}  , 
    D_x \td{c}^{l,n+1} \Big\rangle   
\\
  \ge & 
    \|  D_x \td{c}^{l,n+1} \|_2^2 
   -  \| A_x \td{c}^{l,n+1}  \|_2 \cdot \| D_x e^{q^l \psi^n}  \|_\infty \cdot \| e^{-q^l \psi^n}  \|_\infty  
    \cdot \| D_x \td{c}^{l,n+1} \|_2    
\\
  \ge & 
   \|  D_x \td{c}^{l,n+1} \|_2^2 
   -  \tilde{C}_4^2 \| \td{c}^{l,n+1}  \|_2 \cdot \| D_x \td{c}^{l,n+1} \|_2  ,     
\end{aligned}  
  \label{convergence-4-1}    
\end{equation} 
in which the preliminary $\| \cdot \|_\infty$ bound~\eqref{prelim est-0-1} (in Proposition~\ref{prop: prelim est}) has been applied in the last step. Similar inequalities could be derived in the $y$ and $z$ directions, respectively: 
\begin{align} 
  & 
   \Big\langle e^{-q^l \psi^n} D_y (\td{c}^{l,n+1} e^{q^l \psi^n}),  D_y \td{c}^{l,n+1} \Big\rangle   
  \ge \|  D_y \td{c}^{l,n+1} \|_2^2 
   -  \tilde{C}_4^2 \| \td{c}^{l,n+1}  \|_2 \cdot \| D_y \td{c}^{l,n+1} \|_2  , 
  \label{convergence-4-2}    
\\
  &  
   \Big\langle e^{-q^l \psi^n} D_z (\td{c}^{l,n+1} e^{q^l \psi^n}),  D_z \td{c}^{l,n+1} \Big\rangle   
  \ge \|  D_z \td{c}^{l,n+1} \|_2^2 
   -  \tilde{C}_4^2 \| \td{c}^{l,n+1}  \|_2 \cdot \| D_z \td{c}^{l,n+1} \|_2  .  
  \label{convergence-4-3}    
\end{align} 
As a result, the following estimate becomes available for the nonlinear inner product on the left hand side:  
\begin{equation} 
\begin{aligned} 
  & 
   \Big\langle e^{-q^l\psi^n} \nabla_h (\td{c}^{l,n+1} e^{q^l\psi^n}),  \nabla_h \td{c}^{l,n+1} \Big\rangle   
\\
  \ge & 
   \|  \nabla_h \td{c}^{l,n+1} \|_2^2 
   -  \tilde{C}_4^2 \| \td{c}^{l,n+1}  \|_2 \cdot ( \| D_x \td{c}^{l,n+1} \|_2  
   + \| D_y \td{c}^{l,n+1} \|_2 + \| D_z \td{c}^{l,n+1} \|_2 ) 
\\
   \ge & 
     \|  \nabla_h \td{c}^{l,n+1} \|_2^2 
   -  \sqrt{3} \tilde{C}_4^2 \| \td{c}^{l,n+1}  \|_2 \cdot  \| \nabla_h \td{c}^{l,n+1} \|_2 
\\
   \ge & 
     \|  \nabla_h \td{c}^{l,n+1} \|_2^2 
   -  \frac14 \| \nabla_h \td{c}^{l,n+1} \|_2^2 - 3 \tilde{C}_4^4 \| \td{c}^{l,n+1}  \|_2^2  
   = \frac34 \|  \nabla_h \td{c}^{l,n+1} \|_2^2 
    - 3 \tilde{C}_4^4 \| \td{c}^{l,n+1}  \|_2^2  .      
\end{aligned}  
  \label{convergence-4-4}    
\end{equation} 

The first nonlinear inner product term on the right hand side could be analyzed with the help of the $W_h^{1,\infty}$ regularity assumption~\eqref{assumption:W1-infty bound-2}: 
\begin{equation} 
\begin{aligned} 
  & 
    - \Big\langle (e^{-q^l\check{\psi}^n} - e^{-q^l\psi^n} ) \nabla_h (\check{c}^{l,n+1} e^{q^l\check{\psi}^n} ) ,  \nabla_h \td{c}^{l,n+1} \Big\rangle  
\\
 &  \le 
  \| \nabla_h (\check{c}^{l,n+1} e^{q^l\check{\psi}^n} )  \|_\infty 
  \cdot \| e^{-q^l\check{\psi}^n} - e^{-q^l\psi^n}  \|_2 
  \cdot  \| \nabla_h \td{c}^{l,n+1} \|_2 
\\
 &  \le 
  C_0^* \cdot \tilde{C}_5 \| \td{\psi}^n \|_2   
  \cdot  \| \nabla_h \td{c}^{l,n+1} \|_2  
  \le \frac14 \| \nabla_h \td{c}^{l,n+1} \|_2^2     
  + ( C_0^* )^2 \tilde{C}_5^2 \| \td{\psi}^n \|_2^2 ,   \label{convergence-5}    
\end{aligned} 
\end{equation} 
in which the preliminary $\| \cdot \|_2$ estimates~\eqref{prelim est-0-2} (in Proposition~\ref{prop: prelim est}) has been recalled.   

To analyze the second nonlinear inner product term on the right hand side, we make use of the expansion identity~\eqref{expansion-1} and observe that 
\begin{align} 
  D_x (\check{c}^{l,n+1} (e^{q^l \check{\psi}^n} - e^{q^l \psi^n} ) )  
  = A_x \check{c}^{l,n+1} \cdot  D_x (e^{q^l \check{\psi}^n} - e^{q^l \psi^n} )   
  +  A_x (e^{q^l \check{\psi}^n} - e^{q^l \psi^n} )  
  \cdot D_x  \check{c}^{l,n+1} .  \label{convergence-6-1} 
\end{align}   
This in turn implies that 
\begin{equation} 
\begin{aligned} 
  & 
  - \Big\langle e^{-q^l \psi^n} D_x (\check{c}^{l,n+1} (e^{q^l \check{\psi}^n} - e^{q^l\psi^n} ) ),  
    D_x \td{c}^{l,n+1} \Big\rangle   
\\
  & =  
   \Big\langle e^{-q^l \psi^n} 
   \Big( A_x \check{c}^{l,n+1} \cdot  D_x (e^{q^l \check{\psi}^n} - e^{q^l \psi^n} )   
  +  A_x (e^{q^l \check{\psi}^n} - e^{q^l \psi^n} ) \cdot D_x  \check{c}^{l,n+1}  \Big)  ,  
    D_x \td{c}^{l,n+1} \Big\rangle   
\\
  & \le  
   \| e^{-q^l \psi^n} \|_\infty 
   \cdot \| \check{c}^{l,n+1} \|_\infty \cdot  \| D_x (e^{q^l \check{\psi}^n} - e^{q^l \psi^n} ) \|_2    
   \cdot \| D_x \td{c}^{l,n+1} \|_2   
\\
    & 
    \quad+ \| e^{-q^l \psi^n} \|_\infty  \cdot \| D_x  \check{c}^{l,n+1}  \|_\infty  
   \cdot \| e^{q^l \check{\psi}^n} - e^{q^l \psi^n} \|_2    
   \cdot \| D_x \td{c}^{l,n+1} \|_2 
\\
  & \le  
  \tilde{C}_4 C_0^* \Big( \tilde{C}_6 \| \td{\psi}^n \|_2 
  + \tilde{C}_7 ( \| \td{\psi}^n \|_2  + \| \nabla_h \td{\psi}^n \|_2 ) \Big)  
   \| D_x \td{c}^{l,n+1} \|_2  
\\
  &  \le  
  \frac14 \| D_x \td{c}^{l,n+1} \|_2^2      
  + 2 \tilde{C}_4^2 (C_0^*)^2 \Big(  ( \tilde{C}_6 + \tilde{C}_7 )^2 \| \td{\psi}^n \|_2^2  
  + \tilde{C}_7^2 \| \nabla_h \td{\psi}^n \|_2^2 \Big) ,           
\end{aligned}  
  \label{convergence-6-2}   
\end{equation}   
in which the preliminary estimates~\eqref{prelim est-0-1}, \eqref{prelim est-0-3} (in Proposition~\ref{prop: prelim est}), as well as the regularity assumption~\eqref{r:eq1}, have been extensively used. The inequalities in the $y$ and $z$ directions could be derived in a similar manner. Then we are able to obtain 
\begin{equation} 
\begin{aligned} 
  & 
  - \Big\langle e^{-q^l \psi^n} \nabla_h (\check{c}^{l,n+1} (e^{q^l \check{\psi}^n} - e^{q^l\psi^n} ) ),  
    \nabla_h \td{c}^{l,n+1} \Big\rangle   
\\
  \le & 
  \frac14 \| \nabla_h \td{c}^{l,n+1} \|_2^2      
  + 6 \tilde{C}_4^2 (C_0^*)^2 \Big(  ( \tilde{C}_6 + \tilde{C}_7 )^2 \| \td{\psi}^n \|_2^2  
  + \tilde{C}_7^2 \| \nabla_h \td{\psi}^n \|_2^2 \Big) .            
\end{aligned}  
  \label{convergence-6-3}   
\end{equation}   

Finally, a substitution of \reff{convergence-2}, \reff{convergence-4-4}, \reff{convergence-5} and \reff{convergence-6-3} into \eqref{convergence-1} leads to
\begin{equation}
\begin{aligned}
 \frac{1}{\dt} ( \|\td{c}^{l,n+1} \|^2_2- \|\td{c}^{l,n} \|_2^2 ) + \frac12 \| \nabla_h \td{c}^{l,n+1} \|_2^2
\leq& \| \tau^{l,n+1}  \|_2^2+ ( 6 \tilde{C}_4^4 + 1 ) \|\td{c}^{l,n+1} \|_2^2 \\
&+ \tilde{C}_8 \| \td{\psi}^n \|_2^2 + \tilde{C}_9 \| \nabla_h \td{\psi}^n \|_2^2 , 
\end{aligned} 
  \label{convergence-7-1} 
\end{equation}
with $\tilde{C}_8 = (C_0^*)^2 \Big( 2 \tilde{C}_5^2 + 12 \tilde{C}_4^2 ( \tilde{C}_6 + \tilde{C}_7 )^2 \Big)$, $\tilde{C}_9 = 12 (C_0^*)^2 \tilde{C}_4^2 \tilde{C}_7^2$. Meanwhile, the discrete elliptic regularity inequality~\eqref{lem1:1} (in Lemma~\ref{Lem1}) indicates that 
\begin{equation} 
\begin{aligned} 
  \tilde{C}_8 \| \td{\psi}^n \|_2^2 + \tilde{C}_9 \| \nabla_h \td{\psi}^n \|_2^2 
  \le & C ( \tilde{C}_8  + \tilde{C}_9  )  \| \Delta_h \td{\psi}^n \|_2^2 
  \le C ( \tilde{C}_8  + \tilde{C}_9  )  \Big( \sum_{l=1}^M q^l \| \td{c}^{l,n} \|_2 \Big)^2  
\\
  \le & 
  C M ( \tilde{C}_8  + \tilde{C}_9  )  \sum_{l=1}^M ( q^l )^2 \| \td{c}^{l,n} \|_2^2 . 
\end{aligned}  
   \label{convergence-7-2}    
\end{equation}   
Its substitution into~\eqref{convergence-7-1} yields 
\begin{equation}
\begin{aligned}
 \frac{1}{\dt} ( \|\td{c}^{l,n+1} \|^2_2- \|\td{c}^{l,n} \|_2^2 ) + \frac12 \| \nabla_h \td{c}^{l,n+1} \|_2^2
\leq& \| \tau^{l,n+1}  \|_2^2+ ( 6 \tilde{C}_4^4 + 1 ) \|\td{c}^{l,n+1} \|_2^2 \\
&  + C M ( \tilde{C}_8  + \tilde{C}_9  )  \sum_{l=1}^M ( q^l )^2 \| \td{c}^{l,n} \|_2^2 . 
\end{aligned} 
  \label{convergence-7-3} 
\end{equation}    
Subsequently, a summation over all ionic species gives
\begin{equation}
\begin{aligned} 
  & 
\sum_{l=1}^M \Big( \frac{1}{\dt} ( \|\td{c}^{l,n+1} \|^2_2 - \|\td{c}^{l,n} \|_2^2 ) 
 + \frac12 \|\nabla_h \td{c}^{l,n+1} \|_2^2 \Big) 
 \\
\le & \sum_{l=1}^M \Big( \|\tau^{l,n+1} \|_2^2 + ( 6 \tilde{C}_4^4 + 1 )
 \|\td{c}^{l,n+1} \|_2^2 \Big) 
   + C M^2 ( \tilde{C}_8  + \tilde{C}_9  )  \sum_{l=1}^M ( q^l )^2 \| \td{c}^{l,n} \|_2^2 .
\end{aligned}
\end{equation}
An application of discrete Gronwall inequality leads to 
\begin{equation}\label{th:eq6}
\sum_{l=1}^M \|\td{c}^{l,n+1} \|_2+ \Big( \dt \sum_{l=1}^M \sum_{k=1}^{n+1} \|\nabla_h \td{c}^{l,k} \|_2^2 \Big)^{\frac{1}{2}} \le \hat{C} (\Delta t^2+h^2),
\end{equation}
based on the higher order truncation error accuracy, $\|\tau^{l,n+1} \|_2\leq C ( \dt^2+h^2),~l=1,\cdots,M$.

With this higher order error estimate, we notice that the a-priori assumption~\eqref{a priori-1} is satisfied at the next time step $t^{n+1}$:
\[
\| \td{c}^{l,n+1} \|_2\le \hat{C} (\Delta t^2+h^2)\leq \Delta t^{\frac{15}{8}}+h^{\frac{15}{8}}, \quad 
1 \le l \le M , 
\]
provided that $\dt$ and $h$ are sufficiently small. Therefore, an induction analysis could be applied. This finishes the higher order convergence analysis.

As a result, the convergence estimate \reff{th:eq1} for the variable $(c^1,\cdots,c^M)$ is a direct consequence of \reff{th:eq6}, combined with the definition \reff{s1:eq0} of the constructed approximate solution $(c^l_N, c^2_N,\cdots,c^M_N)$, as well as the projection estimate \reff{ca:eq1}.

In terms of the convergence estimate for the electric potential $\psi$, we recall the definition for $\td{\psi}^{n+1}$ and observe that
\[
\big\|\td{\psi}^{n+1} \big\|_{H^2_h}\leq C \|\Delta_h\td{\psi}^{n+1} \|_2 \le \frac{C}{\kappa}\sum_{l=1}^M|q^l|\big\| \td{c}^{l,n+1}\big\|_2 \le \hat{C}_2 (\Delta t^2+h^2),
\]
in which the discrete elliptic regularity estimate~\eqref{lem1:1} (in Lemma~\ref{Lem1}) has been applied in the first step. Meanwhile, the following inequality is available: 
\begin{align*} 
  & 
 \|\td{\psi}^{n+1}-e^{n+1}_{\psi} \|_{H^2_h} \le C \|\Delta_h(\td{\psi}^{n+1}-e^{n+1}_{\psi}) \|_2 \le  \hat{C}_3 \dt ,
\\
  & 
  \mbox{since} \quad 
(-\Delta_h)(\td{\psi}^{n+1}-e^{n+1}_{\psi})=\sum_{l=1}^Mq^l(\check{c}^{l,n+1}-c^{l,n+1}_N)=\sum_{l=1}^Mq^l(\Delta t\mathcal{P}_N c^{l,n+1}_{\Delta t}).
\end{align*} 
Finally, we arrive at 
\begin{equation}
\big\|e^{n+1}_{\psi} \big\|_{H^2_h}\leq \big\|\td{\psi}^{n+1} \big\|_{H^2_h}+\big\|\td{\psi}^{n+1}-e^{n+1}_{\psi} \big\|_{H^2_h}\leq \hat{C}_4 (\Delta t+h^2).
\end{equation}
This completes the proof of Theorem \ref{Th:convergence}. \qed 

\subsection{Theoretical justification of energy dissipation} 
By the estimate~\reff{th:eq6}, the $\| \cdot \|_\infty$ error bound for the concentration variable at the next time step is available:
	\begin{eqnarray} 
 \| \tilde{c}^{l,n+1} \|_\infty  \le \frac{C \| \tilde{c}^{l,n+1} \|_2}{h^\frac32} \le   \frac{C ( \dt^\frac{15}{8} + h^\frac{15}{8} ) }{h^\frac32} \le  C ( \dt^\frac{3}{8} + h^\frac{3}{8} ) \le \frac14 , \quad 
  l = 1, 2, \cdots, M . 
  	\label{a priori-8}    
	\end{eqnarray}
Again, with the help of the regularity assumption~\eqref{r:eq1}, an $\ell^\infty$ bound for the numerical solution can be derived at the next time step: 
	\begin{eqnarray} 
\| c^{l,n+1} \|_\infty \le  \| \check{c}^{l,n+1}  \|_\infty +  \| \tilde{c}^{l,n+1} \|_\infty \le \tilde{C}_3  , \quad  l = 1, 2, \cdots, M .   
	\label{a priori-9} 
	\end{eqnarray}         

Therefore, with the $\| \cdot \|_\infty$ bounds~\eqref{a priori-9} and \eqref{a priori-7} for the numerical solutions, $c^{l,n+1}$ (evaluated at the next time step) and $\nabla_h \psi^n$ (evaluated at the previous time step), a lower bound of $\tau^*$ in~\eqref{bound-dt-1} can be established:
\[
\tau^*\geq \tau^*_{\rm min}:= \frac{\kappa}{ \tilde{C}_3  \sum_{l=1}^M |q^l|^2 } e^{-|q^l| h \tilde{C}_3}. 
\]
As a result, the proof of Theorem~\ref{theorem:energy} can be theoretically justified, as long as the time step size satisfies $0 < \dt \le \tau^*_{\rm min}$.

\subsection{Error analysis with other mobility means}  
The convergence analysis presented above focuses on the harmonic mean for the mobility average. We remark that the convergence estimate will still go through for the other options, such as the geometric mean, arithmetic mean, or entropic mean (formulated in~\eqref{mean-1}).

In fact, the asymptotic expansion~\eqref{s1:eq0} and the higher order consistency estimate~\eqref{s1:eq7} take the same form, which in turn give the same error evolutionary equation~\eqref{s1:eq8}. The a-priori assumption~\eqref{a priori-1} and the a-priori estimates~\eqref{a priori-6-1}-\eqref{a priori-7} are not affected by the choices of the mobility average. Proposition~\ref{prop: prelim est} is still valid, in which the inequalities~\eqref{prelim est-0-1} and \eqref{prelim est-0-2} are defined in a similar manner, for different options of $e^{- q^l \psi^n}$ at staggered grid points. 

The only essential difference will be the corresponding estimate~\eqref{convergence-3-3}, in which the exact identity is not valid any more. Meanwhile, the $\| \cdot \|_\infty$ bound~\eqref{a priori-7} for $\psi^n$ implies that 
\begin{equation} 
   e^{-q^l \psi^n_{i+\frac12, j, k} }  \cdot ( A_x e^{q^l \psi^n} )_{i+\frac12, j,k} 
   \ge B_0  > 0 ,  \quad \mbox{$B_0$ dependent on $\tilde{C}_3$} . \label{convergence-3-3-b} 
\end{equation}  
In turn, the corresponding estimates in~\eqref{convergence-4-1} and \eqref{convergence-4-4} become
\begin{equation} 
\begin{aligned} 
  & 
   \Big\langle e^{-q^l\psi^n} D_x (\td{c}^{l,n+1} e^{q^l\psi^n}),  D_x \td{c}^{l,n+1} \Big\rangle   
\\
  = & 
    B_0 \|  D_x \td{c}^{l,n+1} \|_2^2 
   +  \Big\langle A_x \td{c}^{l,n+1}   ( D_x e^{q^l \psi^n}  ) e^{-q^l \psi^n}  , 
    D_x \td{c}^{l,n+1} \Big\rangle   
\\
  \ge & 
    B_0 \|  D_x \td{c}^{l,n+1} \|_2^2 
   -  \| A_x \td{c}^{l,n+1}  \|_2 \cdot \| D_x e^{q^l \psi^n}  \|_\infty \cdot \| e^{-q^l \psi^n}  \|_\infty  
    \cdot \| D_x \td{c}^{l,n+1} \|_2    
\\
  \ge & 
   B_0 \|  D_x \td{c}^{l,n+1} \|_2^2 
   -  \tilde{C}_4^2 \| \td{c}^{l,n+1}  \|_2 \cdot \| D_x \td{c}^{l,n+1} \|_2  ,     
\end{aligned}  
  \label{convergence-4-1-b}    
\end{equation} 
\begin{equation} 
\begin{aligned} 
  & 
   \Big\langle e^{-q^l\psi^n} \nabla_h (\td{c}^{l,n+1} e^{q^l\psi^n}),  \nabla_h \td{c}^{l,n+1} \Big\rangle   
\\
  \ge & 
   B_0 \|  \nabla_h \td{c}^{l,n+1} \|_2^2 
   -  \tilde{C}_4^2 \| \td{c}^{l,n+1}  \|_2 \cdot ( \| D_x \td{c}^{l,n+1} \|_2  
   + \| D_y \td{c}^{l,n+1} \|_2 + \| D_z \td{c}^{l,n+1} \|_2 ) 
\\
   \ge & 
     B_0 \|  \nabla_h \td{c}^{l,n+1} \|_2^2 
   -  \sqrt{3} \tilde{C}_4^2 \| \td{c}^{l,n+1}  \|_2 \cdot  \| \nabla_h \td{c}^{l,n+1} \|_2 
\\
   \ge & 
      B_0 \|  \nabla_h \td{c}^{l,n+1} \|_2^2  - \frac14 B_0  \| \nabla_h \td{c}^{l,n+1} \|_2^2 
     - 3 \tilde{C}_4^4 B_0^{-1} \| \td{c}^{l,n+1}  \|_2^2  
\\
  = & 
    \frac34 B_0 \|  \nabla_h \td{c}^{l,n+1} \|_2^2 
    - 3 \tilde{C}_4^4 B_0^{-1} \| \td{c}^{l,n+1}  \|_2^2  .      
\end{aligned}  
  \label{convergence-4-4-b}    
\end{equation} 

Similarly, the estimates in~\eqref{convergence-5} and \eqref{convergence-6-3} could be derived as 
\begin{equation} 
\begin{aligned} 
  & 
    - \Big\langle (e^{-q^l\check{\psi}^n} - e^{-q^l\psi^n} ) \nabla_h (\check{c}^{l,n+1} e^{q^l\check{\psi}^n} ) ,  \nabla_h \td{c}^{l,n+1} \Big\rangle  
\\
  \le & 
  \| \nabla_h (\check{c}^{l,n+1} e^{q^l\check{\psi}^n} )  \|_\infty 
  \cdot \| e^{-q^l\check{\psi}^n} - e^{-q^l\psi^n}  \|_2 
  \cdot  \| \nabla_h \td{c}^{l,n+1} \|_2 
\\
  \le & 
  C_0^* \cdot \tilde{C}_5 \| \td{\psi}^n \|_2   
  \cdot  \| \nabla_h \td{c}^{l,n+1} \|_2  
  \le \frac14 B_0 \| \nabla_h \td{c}^{l,n+1} \|_2^2     
  + ( C_0^* )^2 \tilde{C}_5^2 B_0^{-1} \| \td{\psi}^n \|_2^2  ,  \label{convergence-5-b}    
\end{aligned} 
\end{equation} 
and
\begin{equation} 
\begin{aligned} 
  & 
  - \Big\langle e^{-q^l \psi^n} \nabla_h (\check{c}^{l,n+1} (e^{q^l \check{\psi}^n} - e^{q^l\psi^n} ) ),  
    \nabla_h \td{c}^{l,n+1} \Big\rangle   
\\
  \le & 
  \frac14 B_0 \| \nabla_h \td{c}^{l,n+1} \|_2^2      
  + 6 \tilde{C}_4^2 (C_0^*)^2 B_0^{-1} \Big(  ( \tilde{C}_6 + \tilde{C}_7 )^2 \| \td{\psi}^n \|_2^2  
  + \tilde{C}_7^2 \| \nabla_h \td{\psi}^n \|_2^2 \Big),            
\end{aligned}  
  \label{convergence-6-3-b}   
\end{equation}   
respectively.

Analogous to~\eqref{convergence-7-1}, we arrive at  
\begin{equation}
\begin{aligned}
 \frac{1}{\dt} ( \|\td{c}^{l,n+1} \|^2_2- \|\td{c}^{l,n} \|_2^2 ) + \frac12 B_0 \| \nabla_h \td{c}^{l,n+1} \|_2^2 \leq& \| \tau^{l,n+1}  \|_2^2+ ( 6 \tilde{C}_4^4 B_0^{-1} + 1 ) \|\td{c}^{l,n+1} \|_2^2 \\
&+ \tilde{C}_8 \| \td{\psi}^n \|_2^2 + \tilde{C}_9 \| \nabla_h \td{\psi}^n \|_2^2 , 
\end{aligned} 
  \label{convergence-7-1-b} 
\end{equation}
with $\tilde{C}_8 = (C_0^*)^2 B_0^{-1} \Big( 2 \tilde{C}_5^2 + 12 \tilde{C}_4^2 ( \tilde{C}_6 + \tilde{C}_7 )^2 \Big)$, $\tilde{C}_9 = 12 (C_0^*)^2 \tilde{C}_4^2 \tilde{C}_7^2 B_0^{-1}$. With similar arguments, we are able to establish the convergence estimate~\eqref{th:eq6}. The technical details are left to interested readers.

\subsection{Convergence estimate on fluxes} 
It is of practical significance to accurately approximate the ionic fluxes in many applications. For instance, the accuracy of ionic fluxes through a transmembrane channel is extremely important to the calculation of the current-voltage characteristic curves. We here consider the convergence estimate of the flux vector, defined as $J^l := \nabla c^l + q^l c^ l\nabla\psi$, $1 \le l \le M$. In more detail, we denote the following profiles: 
\begin{equation} 
\begin{aligned}
  & 
  J_e^l =  \nabla c_e^l + q^l c_e^l \nabla \psi_e ,   \, \, \, \mbox{(the exact profile)} ,   
\\
  & 
   J_N^l =  \nabla_h c_N^l + q^l c_N^l \nabla_h \psi_N ,   \, \,\, \mbox{(the approximate profile)}  ,   
\\
  & 
  J^l =  \nabla_h c^l + q^l c^l \nabla_h \psi ,   \, \, \, \mbox{(the numerical profile)}  . 
\end{aligned} 
  \label{error-flux-1} 
\end{equation} 
By the projection estimate~\eqref{projection-ets-1}, combined with the standard truncation error estimate for the finite-difference spatial discretization, we get 
\begin{align}
  \| J_e^l - J_N^l  \|_2 \le C h^2 .  \label{error-flux-2} 
\end{align} 
In turn, the error grid function for the flux vector is defined as
\begin{equation}\label{error function-2}
e^{n}_{J^l}:= J^{l, n}_N - J^{l, n}  ,  \quad l=1,\cdots,M, \, \, \,  n\in\mathbb{N}^*.
\end{equation} 

An $\ell^2 (0, T; \ell^2)$ error estimate can be established for the flux vector. 
\begin{corollary}  \label{cor:convergence-flux}
Let $e^{n}_{J^l}$ be the error grid functions defined in~\eqref{error function-2}.  Then, under the linear refinement requirement $C_1h\leq \Delta t\leq C_2h$, the following convergence result is available as $\dt, h \to 0$: 
\begin{equation} \label{convergence-flux-0}
  \Big( \Delta t \sum_{l=1}^M \sum_{k=1}^{n} \| e^{k}_{J^l} \|^2_2 \Big)^{\frac{1}{2}} \le C(\Delta t+h^2),
\end{equation}
where the constant $C>0$ is independent of $\dt$ and $h$. 
\end{corollary}

\begin{proof} 
A detailed expansion implies that 
\[
   e^{n}_{J^l} =  \nabla_h e^{l, n} + q^l ( c^{l, n} \nabla_h e_\psi^n 
   + e^{l, n} \nabla_h \psi_N^n ) , \quad l=1,\cdots,M. 
\]
Then
\begin{equation} 
\begin{aligned}  
  \| e^{n}_{J^l} \|_2 \le & \| \nabla_h e^{l, n} \|_2 + | q^l | ( \| c^{l, n} \|_\infty \cdot \| \nabla_h e_\psi^n \|_2 + \| e^{l, n} \|_2 \cdot \| \nabla_h \psi^n_N \|_\infty )   
\\
  \le & 
  \| \nabla_h e^{l, n} \|_2 + \tilde{C}_3 | q^l | ( \| \nabla_h e_\psi^n \|_2 +  \| e^{l, n} \|_2  )  , 
\end{aligned} 
\label{convergence-flux-1} 
\end{equation} 
in which the regularity assumption~\eqref{r:eq1}, the a-priori $\| \cdot \|_{W_h^\infty}$ estimate~\eqref{a priori-7}, and an estimate on $c^{l, n}$ that is derived analogously to \reff{a priori-9} have been applied. This in turn leads to 
\begin{equation} 
\begin{aligned}  
  \Big( \Delta t \sum_{l=1}^M \sum_{k=1}^{n} \| e^{k}_{J^l} \|^2_2 \Big)^{\frac{1}{2}} 
  \le & \Big( \Delta t \sum_{l=1}^M \sum_{k=1}^{n} \| \nabla_h e^{l, k} \|^2_2 \Big)^{\frac{1}{2}}   
\\
  & 
  + C | q^l | \Big( \Delta t \sum_{l=1}^M \sum_{k=1}^{n} 
  ( \| \nabla_h e_\psi^k \|^2_2 + \| e^{l, k} \|^2_2 ) \Big)^{\frac{1}{2}} . 
\end{aligned} 
   \label{convergence-flux-2} 
\end{equation}    
Moreover, with an application of the discrete Sobolev inequality~\eqref{lem1:1} (in Lemma~\ref{Lem1}) and \eqref{th:eq1} (in Theorem~\ref{Th:convergence}), we see that 
\begin{align} 
  & 
  \Big( \Delta t \sum_{l=1}^M \sum_{k=1}^{n} \| \nabla_h e^{l, k} \|^2_2 \Big)^{\frac{1}{2}}   
  \le C (\dt + h^2) ,  \quad \mbox{(by~\eqref{th:eq1})} ,  \label{convergence-flux-3-1}   
\\
  & 
   \| \nabla_h e_\psi^k \|_2 \le C \| \Delta_h e_\psi^k \|_2  \le C \| e_\psi^k \|_{H_h^2} 
   \le C (\dt + h^2) ,  \quad \mbox{(by~\eqref{lem1:1})},  \quad \mbox{so that}  \nonumber 
\\
  & 
   \Big( \Delta t \sum_{l=1}^M \sum_{k=1}^{n} 
   \| \nabla_h e_\psi^k \|^2_2  \Big)^{\frac{1}{2}}  \le C (\dt + h^2) ,   \label{convergence-flux-3-2} 
\\
  & 
   \| e^{l, k} \|_2 \le C (\dt + h^2) ,  \, \, \, \mbox{(by~\eqref{th:eq1}) } ,  \quad 
   \mbox{so that}  \, \, \,  
   \Big( \Delta t \sum_{l=1}^M \sum_{k=1}^{n} 
   \| e^{l, k} \|^2_2  \Big)^{\frac{1}{2}}  \le C (\dt + h^2) .   \label{convergence-flux-3-3}           
\end{align}       
Finally, a substitution of~\eqref{convergence-flux-3-1}-\eqref{convergence-flux-3-3} into \eqref{convergence-flux-2} yields the desired error estimate~\eqref{convergence-flux-0} for the flux vector. This finishes the proof of Corollary~\ref{cor:convergence-flux}. \qed
\end{proof}

\begin{remark} 
Since the gradient operator is involved in the flux vector, an $\ell^\infty (0, T; \ell^2)$ error estimate for $e_{J^l}$ is not directly available. On the other hand, such an optimal rate error estimate in the $\ell^\infty (0, T; \ell^2)$ norm can be established for the flux vector, if an even higher order consistency analysis is performed. The details are left to interested readers. 
\end{remark}

\section{Numerical results} \label{sec:numerical results}
Numerical simulations are conducted to demonstrate numerical accuracy of the developed numerical method and its effectiveness in preserving mass conservation, positivity, and free-energy dissipation. Unless stated otherwise, we take the characteristic concentration $c_0=1M$ and characteristic length $L=1nm$.


\subsection{Accuracy test}
Consider an electrolyte solution with binary symmetric monovalent ions. Let $\kappa=1$. To test the accuracy of the numerical method, we consider the following problem in 2D:
\begin{equation}
\left\{
\begin{aligned}
&\partial_t c^1=\nabla\cdot(\nabla c^1 +c^1 \nabla\psi)+f_1,\\
&\partial_t c^2=\nabla\cdot(\nabla c^2 -c^2 \nabla\psi)+f_2,\\
&-\kappa\Delta \psi=c^1-c^2+\rho^f,
\end{aligned}
\right.
\end{equation}
where the source terms $f_1$, $f_2$, and $\rho^f$, as well as the initial and boundary conditions, are determined by the following exact periodic solution
\begin{equation}
\left\{
\begin{aligned}
&c^1(x,y,t)=e^{-t}\cos(2\pi x)\sin(2\pi y)+2,\\
&c^2(x,y,t)=e^{-t}\cos(2\pi x)\sin(2\pi y)+2,\\
&\psi(x,y,t)=e^{-t}\cos(2\pi x)\sin(2\pi y).
\end{aligned}
\right.
\end{equation}

\begin{table}[H]
\centering
\begin{tabular}{cccccccc}
\hline\hline
&$h$ & $\ell^\infty$ error in $c^1$ & Order&$\ell^\infty$ error in $c^2$ & Order&$\ell^\infty$ error in $\psi$ & Order\\
\hline
 \multirow{4}{*}{Harmonic mean}
&$\frac{1}{50}$&2.00e-03 &- &1.80e-03 &- & 1.20e-03&-\\
&$\frac{1}{60}$&1.40e-03 &2.01 &1.20e-03 &2.01 & 8.37e-04&2.01\\
&$\frac{1}{70}$&1.00e-03 &1.99 &9.19e-04 &1.98 & 6.16e-04&1.99\\
&$\frac{1}{80}$&7.65e-04 &2.01 &7.03e-04 &2.01 & 4.71e-04&2.00\\
&$\frac{1}{90}$&6.05e-04 &2.00 &5.56e-04 &1.99& 3.73e-04&1.99\\
\hline
\multirow{4}{*}{Geometric mean}
&$\frac{1}{50}$&2.00e-03 &- &1.80e-03 &- & 1.20e-03&-\\
&$\frac{1}{60}$&1.40e-03 &2.01 &1.20e-03 &2.01 & 8.37e-04&2.01\\
&$\frac{1}{70}$&1.00e-03 &1.99 &9.19e-04 &1.98 & 6.16e-04&1.99\\
&$\frac{1}{80}$&7.65e-04 &2.01 &7.03e-04 &2.01 & 4.71e-04&2.00\\
&$\frac{1}{90}$&6.05e-04 &2.00 &5.56e-04 &1.99& 3.73e-04&1.99\\  
\hline
\multirow{4}{*}{Arithmetic mean}
&$\frac{1}{50}$&4.90e-03 &- &2.40e-03 &- & 1.10e-03&-\\
&$\frac{1}{60}$&3.40e-03 &2.01 &1.60e-03 &2.01 & 7.84e-04&2.01\\
&$\frac{1}{70}$&2.50e-03 &1.99 &1.20e-03 &1.99 & 5.77e-04&1.99\\
&$\frac{1}{80}$&1.90e-03 &2.01 &9.24e-04 &2.01 & 4.41e-04&2.00\\
&$\frac{1}{90}$&1.50e-03 &1.99 &7.30e-04 &1.99& 3.49e-04&1.99\\    
\hline
\multirow{4}{*}{Entropic mean}
&$\frac{1}{50}$&1.20e-03 &- &2.00e-03 &- & 1.20e-03&-\\
&$\frac{1}{60}$&8.04e-04 &2.02 &1.40e-03 &2.01 & 8.37e-04&2.01\\
&$\frac{1}{70}$&5.93e-04 &1.98 &1.00e-04 &1.99 & 6.03e-04&1.99\\
&$\frac{1}{80}$&4.53e-04 &2.01 &7.77e-04 &2.00 & 4.61e-04&2.00\\
&$\frac{1}{90}$&3.59e-04 &1.99 &6.15e-04 &1.99& 3.65e-04&1.99\\   
\hline\hline
\end{tabular}
\caption{Numerical error and convergence order of numerical solutions at time $T=0.1$ with a mesh ratio $\Delta t=h^2$.}
\label{t:OrderTab}
\end{table}
To probe the discretization error, we apply the proposed numerical method to solve the problem using various spatial step size $h$ with a fixed mesh ratio $\Delta t=h^2$. Table \ref{t:OrderTab} lists the $\ell^\infty$ errors and convergence orders for ionic concentration and electrostatic potential at time $T=0.1$. Obviously, we can observe that the error decreases as the mesh refines, and the convergence orders for ion concentrations and the electric potential are both perfectly about $2$. This indicates that the semi-implicit scheme~\reff{FuPNP} has expected convergence rate, i.e., first-order and second-order accurate in time and spatial discretization, respectively. Notice that the mesh ratio, $\Delta t=h^2$, adopted here is for the purpose of numerical accuracy test, not for the enforcement of the numerical stability or positivity.

\subsection{Properties test}
Consider a closed, neutral system that consists of symmetric monovalent ions with $\kappa=1e-3$. The fixed charge density
\[
\begin{aligned}
\rho^f(x,y)=&-e^{-100\left[(x-\frac{1}{4})^2+(y-\frac{1}{4})^2\right]}
+e^{-100\left[(x-\frac{1}{4})^2+(y-\frac{3}{4})^2\right]}\\
&+e^{-100\left[(x-\frac{3}{4})^2+(y-\frac{1}{4})^2\right]}
-e^{-100\left[(x-\frac{3}{4})^2+(y-\frac{3}{4})^2\right]}
\end{aligned}
\]
is prescribed to approximate two positive and two negative point charges located in four quadrants, using Gaussian functions with small local supports. The initial distributions are given by 
\begin{equation}\label{InitCons}
\begin{aligned}
c^1(x,y, 0)=0.1~~\mbox{and}~~c^2(x,y, 0)=0.1.
\end{aligned}
\end{equation}

\begin{figure}[H]
\centering
\includegraphics[scale=.65]{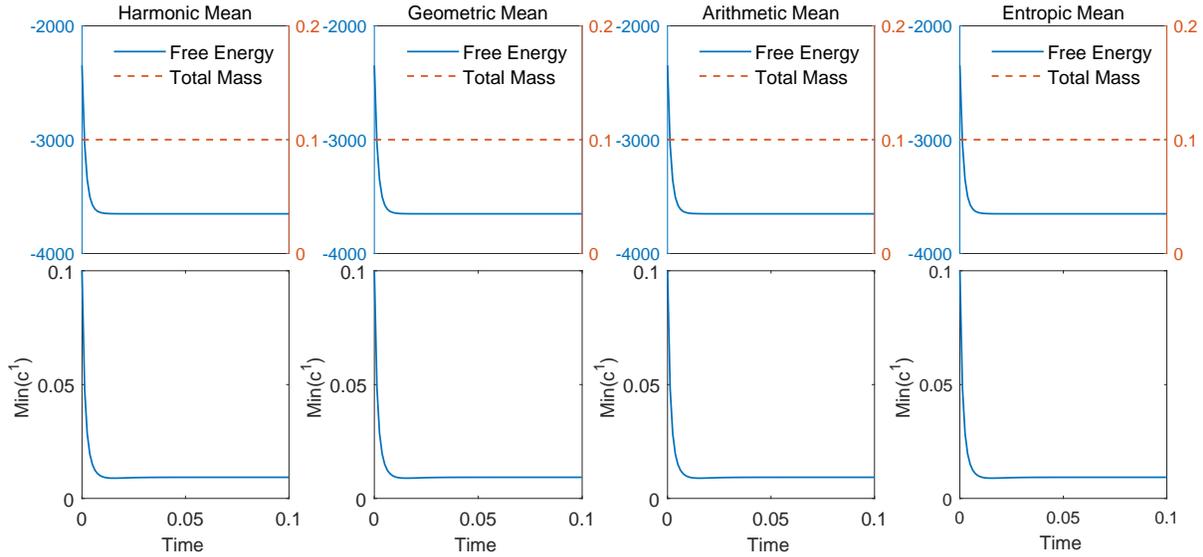}
\caption{The evolution of the discrete energy $F_h$, mass of cations, and minimum concentration of cations with $N=80$ and $\Delta t=h/10$.}
\label{f:MassEnergy}
\end{figure}
In this test, we assess the performance of the Slotboom scheme in preserving physical properties with four different means: the harmonic mean, geometric mean, arithmetic mean, and entropic mean. As displayed in Fig.~\ref{f:MassEnergy}, the Slotboom schemes with four different means perfectly conserve the total ion concentration. The profiles of discrete free energy~\reff{Fh} are shown to decay monotonically and robustly, and numerical solutions of concentrations remain positive for all the time.  Such results are consistent with the theoretical analysis.

\section{Conclusions}  \label{sec:conclusion}
Structure-preserving numerical methods for the Poisson--Nernst--Planck (PNP) equations have attracted lots of attention recently. Based on the Slotboom transformation, a class of numerical methods that can be proved to preserve mass conservation, ionic concentration positivity, and free-energy dissipation at discrete level have been derived in literature. However, rigorous convergence analysis for such structure-preserving schemes has been still open. This work has provided optimal rate convergence analysis for such structure-preserving schemes based on the Slotboom reformulation.  Different options of mobility average at the staggered mesh points have been considered for spatial finite-difference discretization, such as the harmonic mean, geometric mean, arithmetic mean, and entropic mean. A semi-implicit temporal discretization has been employed, therefore only a non-constant coefficient, positive-definite linear system has to be solved at each time step. A higher order asymptotic expansion has been used in the consistency analysis, so that the discrete maximum norm of the concentration variables can be controlled.  The harmonic mean for the mobility average has been taken in the convergence estimate for simplicity, while the desired error estimate for other options of mobility average has been elaborated as well with more technical details.  An optimal rate convergence analysis for the ionic concentrations, electric potential, and fluxes has been established, which is the first such result for the structure-preserving numerical schemes based on the Slotboom reformulation. With such convergence analysis, the conditional energy dissipation analysis that relies on the maximum norm bounds of the concentration and the gradient of the electric potential has been further validated. Numerical results have also been presented to demonstrate the accuracy and structure-preserving performance of the numerical schemes.

We now discuss several possible further refinements of our work. The current convergence analysis could be extended to consider the fully implicit schemes that can unconditionally preserve energy dissipation. The unconditional energy dissipation would help establish upper bounds that are useful in the error estimate.  In addition, it is of practical interest to consider convergence analysis of structure-preserving schemes for \emph{modified} PNP equations based on the Slotboom reformulation.  For modified PNP equations, the corresponding structure-preserving numerical methods can be analogously proposed based on \reff{NPe} with
\[
S^l=q^l \psi + \mu^{l}_{\rm excess},
\]
where the excess chemical potential $\mu^l_{\rm excess}$ can be introduced to consider various effects that are neglected by the classical PNP models. For instance,  the excess chemical potential can describe ionic steric effects either by hard-sphere repulsion described by the fundamental measure theory~\cite{Roth:JPC:2002,Wu:JCP:2002} or by the incorporation of the entropy of solvent molecules~\cite{Li:N:2009, ZhouWangLi_PRE11}. 
\vskip 5mm
\noindent{\bf Acknowledgements.}
This work is supported in part by the Natural Science Foundation of Jiangsu Province BK20210443, High level personnel project of Jiangsu Province 1142024031211190, National Natural Science Foundation of China 12101264 (J. Ding), 21773165, Natural Science Foundation of Jiangsu Province BK20200098, China, and Shanghai Science and Technology Commission 21JC1403700 (S. Zhou), and NSF DMS-2012669 (C. Wang).
 
\appendix
\numberwithin{equation}{section}
\makeatletter
\newcommand{\section@cntformat}{Appendix\thesection:\ }
\makeatother
\section[Appendix]{Appendix A. Proof of Theorem \ref{theorem:energy}}\label{Ap:A}
For completeness, we present the following proof which is based on the proof of Theorem 3.4 in the work~\cite{LiuMaimaiti2021}, with  modifications due to the entropic mean under consideration. 
\begin{proof}
By~\reff{Fh},  we have
\begin{equation}
\begin{aligned}
F_h^{n+1}-F_h^n=&\sum_{l=1}^M\ciptwo{c^{l,n+1}}{\log c^{l,n+1}}-\ciptwo{c^{l,n}}{\log c^{l,n}}\\
&+\frac{1}{2}\ciptwo{\rho^{n+1}}{\psi^{n+1}}-\frac{1}{2}\ciptwo{\rho^n}{\psi^n}\\
:=&-\Delta t I_1+(\Delta t)^2 I_2,
\end{aligned}
\end{equation}
where
\begin{equation}
\begin{aligned}
I_1&=-\frac{1}{\Delta t}\sum_{l=1}^M\left[\ciptwo{c^{l,n+1}}{\log c^{l,n+1}}-\ciptwo{c^{l,n}}{\log c^{l,n}}+\ciptwo{c^{l,n+1}-c^{l,n}}{q^l\psi^n}\right]\\
I_2&=\frac{1}{2(\Delta t)^2}\left[\ciptwo{\rho^{n+1}}{\psi^{n+1}}-\ciptwo{\rho^{n}}{\psi^{n}}-2\ciptwo{\rho^{n+1}-\rho^n}{\psi^n} \right].\\
\end{aligned}
\end{equation}
Thus, the energy dissipation inequality is satisfied if
\[
\Delta t \leq \tau^*\leq \frac{I_1}{2I_2}.
\]

By the discretization scheme\reff{FuPNP} and periodic boundary conditions, we have
\[
\begin{aligned}
I_1\geq -\sum_{l=1}^M\ciptwo{\frac{c^{l,n+1}-c^{l,n}}{\Delta t}}{\log g^{l,n+1}}
\geq\sum_{l=1}^M\Big\langle e^{-q^l\psi^n}\nabla_h g^{l,n+1}, \nabla_h\log g^{l,n+1} \Big\rangle
\geq 0,
\end{aligned}
\]
where $g^{l,n+1}=c^{l,n+1}e^{q^l\psi^n}$, the summation by parts, and $(\log X-\log Y)(X-Y)>0$ for $X, Y>0$ have been used.
In addition, one can easily verify that
\[
\ciptwo{\rho^{n+1}}{\psi^n}=\ciptwo{\rho^n}{\psi^{n+1}},
\]
which further implies that
\[
I_2=\frac{1}{2(\Delta t)^2}\ciptwo{\rho^{n+1}-\rho^n}{ \psi^{n+1}-\psi^n}.
\]
By summation by parts, one can show by the Cauchy inequality that
\begin{equation}\label{Apc:1}
\begin{aligned}
I_2=&-\sum_{l=1}^M\frac{q^l}{2\Delta t}\Big\langle e^{-q^l\psi^n}\nabla_h g^{l,n+1}, \nabla_h(\psi^{n+1}-\psi^n) \Big\rangle\\
\leq&\sum_{l=1}^M\frac{|q^l|}{2\Delta t}\big\| e^{-q^l\psi^n}\nabla_h g^{l,n+1}\big\|_2\big\|\nabla_h(\psi^{n+1}-\psi^n) \big\|_2.
\end{aligned}
\end{equation}

On the other hand, it follows from the discrete Poisson's equation~\reff{DPssn} that
\begin{equation}\label{Apc:2}
\begin{aligned}
I_2=\frac{\kappa}{2(\Delta t)^2}\ciptwo{\Delta_h\psi^n-\Delta_h\psi^{n+1}}{\psi^{n+1}-\psi^n }
=\frac{\kappa}{2(\Delta t)^2}\big\|\nabla_h(\psi^{n+1}-\psi^n) \big\|_2^2.
\end{aligned}
\end{equation}

Combination of \reff{Apc:1} and \reff{Apc:2} yields
\[
\begin{aligned}
\big\|\nabla_h(\psi^{n+1}-\psi^n) \big\|_2^2\leq &\frac{\Delta t^2}{\kappa^2}\left(\sum_{l=1}^M |q^l| \big\| e^{-q^l\psi^n}\nabla_h g^{l,n+1}\big\|_2\right)^2\\
\leq&\frac{\Delta t^2}{\kappa^2}\sum_{l=1}^M |q^l|^2 \sum_{l=1}^M\big\| e^{-q^l\psi^n}\nabla_h g^{l,n+1}\big\|_2^2.\\
\end{aligned}
\]
Thus, we have
\[
I_2\leq C_0\sum_{l=1}^M\big\| e^{-q^l\psi^n}\nabla_h g^{l,n+1}\big\|_2^2,\\
\]
where $C_0=\frac{\sum_{l=1}^M (q^l)^2}{2\kappa}$. Therefore,
\[
\begin{aligned}
\frac{I_1}{2I_2}
\geq& \frac{\sum_{l=1}^M\Big\langle e^{-q^l\psi^n}\nabla_h g^{l,n+1}, \nabla_h\log g^{l,n+1}\Big\rangle}{2C_0\sum_{l=1}^M\big\| e^{-q^l\psi^n}\nabla_h g^{l,n+1}\big\|_2^2}\\
\geq&\frac{1}{2C_0}\min_{l,i,j,k}\left\{\frac{\log g^{l,n+1}_{i+1,j,k} -\log g^{l,n+1}_{i,j,k}}{e^{-q^l\psi^n_{i+\frac{1}{2},j,k}}(g^{l,n+1}_{i+1,j,k}-g^{l,n+1}_{i,j,k})}, \quad\frac{\log g^{l,n+1}_{i,j+1,k}-\log g^{l,n+1}_{i,j,k}}{e^{-q^l\psi^n_{i,j+\frac{1}{2},k}}(g^{l,n+1}_{i,j+1,k}-g^{l,n+1}_{i,j,k})},\right.\\
&\left.\qquad\qquad\quad\frac{\log g^{l,n+1}_{i,j,k+1}-\log g^{l,n+1}_{i,j,k}}{e^{-q^l\psi^n_{i,j,k+\frac{1}{2}}}(g^{l,n+1}_{i,j,k+1}-g^{l,n+1}_{i,j,k})} \right\}\\
\geq&\frac{1}{2C_0}\min_{l,i,j,k}\left\{\frac{1}{(g^{l,n+1}_{i+1,j,k})^{\theta}(g^{l,n+1}_{i,j,k})^{(1-\theta)}e^{-q^l\psi^n_{i+\frac{1}{2},j,k}}},
\quad\frac{1}{(g^{l,n+1}_{i,j+1,k})^{\theta}(g^{l,n+1}_{i,j,k})^{(1-\theta)}e^{-q^l\psi^n_{i,j+\frac{1}{2},k}}}, \right.\\
&\left.\qquad\qquad\quad\frac{1}{(g^{l,n+1}_{i,j,k+1})^{\theta}(g^{l,n+1}_{i,j,k})^{(1-\theta)}e^{-q^l\psi^n_{i,j,k+\frac{1}{2}}}} \right\},
\end{aligned}
\]
where 
\[
\frac{X-Y}{e^X-e^Y}=\frac{1}{e^{\theta X+(1-\theta)Y}}~~\mbox{for}~~\theta\in(0,1)
\]
has been used in the second inequality.
Using the entropic mean in the $x$ direction
$$e^{-q^l\psi^n_{i+\frac{1}{2},j,k}}=\frac{q^l(\psi^n_{i+1,j,k}-\psi^n_{i,j,k})}{e^{q^l\psi^n_{i+1,j,k}}-e^{q^l\psi^n_{i,j,k}}},$$
we obtain
\[
\begin{aligned}
\frac{1}{(g^{l,n+1}_{i+1,j,k})^{\theta}(g^{l,n+1}_{i,j,k})^{(1-\theta)}e^{-q^l\psi^n_{i+\frac{1}{2},j,k}}}
=&\frac{e^{q^l\psi^n_{i+1,j,k}}-e^{q^l\psi^n_{i,j,k}}}{(c^{l,n+1}_{i,j,k})^{1-\theta}(c^{l,n+1}_{i+1,j,k})^{\theta}e^{q^l\theta\psi^n_{i+1,j,k}}e^{q^l(1-\theta)\psi^n_{i,j,k}}q^l(\psi^n_{i+1,j,k}-\psi^n_{i,j,k})}\\
=&\frac{e^{q^l(1-\theta)(\psi^n_{i+1,j,k}-\psi^n_{i,j,k})}-e^{q^l\theta(\psi^n_{i,j,k}-\psi^n_{i+1,j,k})}}{q^l(\psi^n_{i+1,j,k}-\psi^n_{i,j,k})(c^{l,n+1}_{i,j,k})^{1-\theta}(c^{l,n+1}_{i+1,jk})^{\theta}}\\
\geq&\frac{e^{-|q^l(\psi^n_{i+1,j,k}-\psi^n_{i,j,k}) |}}{M_1}\\
\geq&\frac{e^{-h |q^l| \|\nabla_h\psi^n\|_{\infty}}}{M_1},\\
\end{aligned}
\]
where $M_1=\underset{1\leq l\leq M}{\max}\|c^{l,n+1} \|_{\infty}$. Similarly, we have
\begin{equation}
\begin{aligned}
&\frac{1}{(g^{l,n+1}_{i,j+1,k})^{\theta}(g^{l,n+1}_{i,j,k})^{1-\theta}e^{-q^l\psi^n_{i,j+\frac{1}{2},k}}}\geq\frac{e^{-|q^l|h\|\nabla_h\psi^n\|_{\infty}}}{M_1},\\
&\frac{1}{(g^{l,n+1}_{i,j,k+1})^{\theta}(g^{l,n+1}_{i,j,k})^{1-\theta}e^{-q^l\psi^n_{i,j,k+\frac{1}{2}}}}\geq\frac{e^{-|q^l|h\|\nabla_h\psi^n\|_{\infty}}}{M_1}.
\end{aligned}
\end{equation}
Therefore, we obtain
\[
\frac{I_1}{2I_2}\geq\frac{e^{-|q^l|h\|\nabla_h\psi^n\|_{\infty}}}{2C_0M_1}.
\]
This completes the proof.
\qed

\end{proof}

\bibliographystyle{plain}
\bibliography{PNP}

\begin{thebibliography}{10}

\bibitem{baskaran13b}
A.~Baskaran, J.~Lowengrub, C.~Wang, and S.~M. Wise.
\newblock Convergence analysis of a second order convex splitting scheme for
  the modified phase field crystal equation.
\newblock {\em SIAM J. Numer. Anal.}, 51:2851--2873, 2013.

\bibitem{BTA:PRE:04}
M.~Bazant, K.~Thornton, and A.~Ajdari.
\newblock Diffuse-charge dynamics in electrochemical systems.
\newblock {\em Phys. Rev. E}, 70(2):021506, 2004.

\bibitem{Chatard2014}
M.~Chatard, C.~Chainais-Hillairet, and M.~H. Vignal.
\newblock Study of a finite volume scheme for the drift-diffusion system.
  {A}symptotic behavior in the quasi-neutral limit.
\newblock {\em SIAM J. Numer. Anal.}, 52:1666--1691, 2014.

\bibitem{CCAO14}
J.~Chaudhry, J.~Comer, A.~Aksimentiev, and L.~Olson.
\newblock A stabilized finite element method for modified
  {P}oisson-{N}ernst-{P}lanck equations to determine ion flow through a
  nanopore.
\newblock {\em Commun. Comput. Phys.}, 15:93--125, 2014.

\bibitem{DSWZhou_CICP18}
J.~Ding, H.~Sun, Z.~Wang, and S.~Zhou.
\newblock Computational study on hysteresis of ion channels: Multiple solutions
  to steady-state {Poisson--Nernst--Planck} equations.
\newblock {\em Commun. Comput. Phys.}, 23(5):1549--1572, 2018.

\bibitem{DingWangZhou_NMTMA19}
J.~Ding, C.~Wang, and S.~Zhou.
\newblock Optimal rate convergence analysis of a second order numerical scheme
  for the {Poisson--Nernst--Planck} system.
\newblock {\em Numer. Math. Theor. Meth. Appl.}, 12:607--626, 2019.

\bibitem{DingWangZhou_JCP2019}
J.~Ding, Z.~Wang, and S.~Zhou.
\newblock Positivity preserving finite difference methods for
  {P}oisson-{N}ernst-{P}lanck equations with steric interactions: Application
  to slit-shaped nanopore conductance.
\newblock {\em J. Comput. Phys.}, 397:108864, 2019.

\bibitem{DingWangZhou_JCP2020}
J.~Ding, Z.~Wang, and S.~Zhou.
\newblock Structure-preserving and efficient numerical methods for ion
  transport.
\newblock {\em J. Comput. Phys.}, 418:109597, 2020.

\bibitem{E95}
W.~E and J.~Liu.
\newblock Projection method {I}: {Convergence} and numerical boundary layers.
\newblock {\em SIAM J. Numer. Anal.}, 32:1017--1057, 1995.

\bibitem{E02}
W.~E and J.~Liu.
\newblock Projection method {III}. {Spatial} discretization on the staggered
  grid.
\newblock {\em Math. Comp.}, 71:27--47, 2002.

\bibitem{AMEKLL14}
A.~Flavell, M.~Machen, R.~Eisenberg, J.~Kabre, C.~Liu, and X.~Li.
\newblock A conservative finite difference scheme for
  {P}oisson-{N}ernst-{P}lanck equations.
\newblock {\em J. Comput. Electron.}, 13:235--249, 2014.

\bibitem{GaoHe_JSC17}
H.~Gao and D.~He.
\newblock Linearized conservative finite element methods for the
  {Nernst--Planck--Poisson} equations.
\newblock {\em J. Sci. Comput.}, 72:1269--1289, 2017.

\bibitem{guan14a}
Z.~Guan, C.~Wang, and S.~M. Wise.
\newblock A convergent convex splitting scheme for the periodic nonlocal
  {Cahn-Hilliard} equation.
\newblock {\em Numer. Math.}, 128:377--406, 2014.

\bibitem{guo16}
J.~Guo, C.~Wang, S.~M. Wise, and X.~Yue.
\newblock An {$H^2$} convergence of a second-order convex-splitting, finite
  difference scheme for the three-dimensional {Cahn-Hilliard} equation.
\newblock {\em Commun. Math. Sci.}, 14:489--515, 2016.

\bibitem{guo2021}
J.~Guo, C.~Wang, S.~M. Wise, and X.~Yue.
\newblock An improved error analysis for a second-order numerical scheme for
  the {Cahn-Hilliard} equation.
\newblock {\em J. Comput. Appl. Math.}, 388:113300, 2021.

\bibitem{HuHuang_NM20}
J.~Hu and X.~Huang.
\newblock A fully discrete positivity-preserving and energy-dissipative finite
  difference scheme for {Poisson--Nernst--Planck} equations.
\newblock {\em Numer. Math.}, 145:77--115, 2020.

\bibitem{Li:N:2009}
B.~Li.
\newblock Continuum electrostatics for ionic solutions with nonuniform ionic
  sizes.
\newblock {\em Nonlinearity}, 22:811--833, 2009.

\bibitem{LiX2021}
X.~Li, Z.~Qiao, and C.~Wang.
\newblock Convergence analysis for a stabilized linear semi-implicit numerical
  scheme for the nonlocal {Cahn-Hilliard} equation.
\newblock {\em Math. Comp.}, 90:171--188, 2021.

\bibitem{LiuC2021a}
C.~Liu, C.~Wang, S.~M. Wise, X.~Yue, and S.~Zhou.
\newblock A positivity-preserving, energy stable and convergent numerical
  scheme for the {Poisson-Nernst-Planck} system.
\newblock {\em Math. Comp.}, 90:2071--2106, 2021.

\bibitem{LiuWang_Sub21}
C.~Liu, C.~Wang, S.~M. Wise, X.~Yue, and S.~Zhou.
\newblock A second order accurate numerical method for the
  {Poisson--Nernst--Planck} system in the energetic variational formulation.
\newblock {\em Submitted}, 2021.

\bibitem{LiuMaimaiti2021}
H.~Liu and W.~Maimaitiyiming.
\newblock Efficient, positive, and energy stable schemes for multi-d
  {P}oisson–{N}ernst–{P}lanck systems.
\newblock {\em J. Sci. Comput.}, 87(3):1--36, 2021.

\bibitem{LW14}
H.~Liu and Z.~Wang.
\newblock A free energy satisfying finite difference method for
  {P}oisson-{N}ernst-{P}lanck equations.
\newblock {\em J. Comput. Phys.}, 268:363--376, 2014.

\bibitem{LW17}
H.~Liu and Z.~Wang.
\newblock A free energy satisfying discontinuous {G}alerkin method for
  one-dimensional {P}oisson-{N}ernst-{P}lanck systems.
\newblock {\em J. Comput. Phys.}, 328:413--437, 2017.

\bibitem{LiuShu_SCM16}
Y.~Liu and C.~Shu.
\newblock Analysis of the local discontinuous {G}alerkin method for the
  drift-diffusion model of semiconductor devices.
\newblock {\em Sci. China Math.}, 59:115--140, 2016.

\bibitem{LHMZ10}
B.~Lu, M.~Holst, J.~McCammon, and Y.~Zhou.
\newblock {Poisson-Nernst-Planck} equations for simulating biomolecular
  diffusion-reaction processes {I}: {F}inite element solutions.
\newblock {\em J. Comput. Phys.}, 229(19):6979--6994, 2010.

\bibitem{PMarkowich_Book}
P.~Markowich.
\newblock {\em The Stationary Semiconductor Device Equations}.
\newblock Springer-Verlag, New York, 1986.

\bibitem{MXL16}
M.~Metti, J.~Xu, and C.~Liu.
\newblock Energetically stable discretizations for charge transport and
  electrokinetic models.
\newblock {\em J. Comput. Phys.}, 306:1--18, 2016.

\bibitem{Gibou_JCP14}
M.~Mirzadeh and F.~Gibou.
\newblock A conservative discretization of the {P}oisson-{N}ernst-{P}lanck
  equations on adaptive cartesian grids.
\newblock {\em J. Comput. Phys.}, 274:633--653, 2014.

\bibitem{ProhlSchmuck09}
A.~Prohl and M.~Schmuck.
\newblock Convergent discretizations for the {Nernst--Planck--Poisson} system.
\newblock {\em Numer. Math.}, 111:591--630, 2009.

\bibitem{QianWangZhou_JCP21}
Y.~Qian, C.~Wang, and S.~Zhou.
\newblock A positive and energy stable numerical scheme for the
  {Poisson--Nernst--Planck--Cahn--Hilliard} equations with steric interactions.
\newblock {\em J. Comput. Phys.}, 426:109908, 2021.

\bibitem{QianWangZhou_JCP2019}
Y.~Qian, Z.~Wang, and S.~Zhou.
\newblock A conservative, free energy dissipating, and positivity preserving
  finite difference scheme for multi-dimensional nonlocal {F}okker--{P}lanck
  equation.
\newblock {\em J. Comput. Phys.}, 386:22--36, 2019.

\bibitem{Roth:JPC:2002}
R.~Roth, R.~Evans, A.~Lang, and G.~Kahl.
\newblock Fundamental measure theory for hard-sphere mixtures revisited: the
  white bear version.
\newblock {\em J. Phys.: Condens. Matter}, 14(46):12063, 2002.

\bibitem{STWW03}
R.~Samelson, R.~Temam, C.~Wang, and S.~Wang.
\newblock Surface pressure {Poisson} equation formulation of the primitive
  equations: {Numerical} schemes.
\newblock {\em SIAM J. Numer. Anal.}, 41:1163--1194, 2003.

\bibitem{STWW07}
R.~Samelson, R.~Temam, C.~Wang, and S.~Wang.
\newblock A fourth order numerical method for the planetary geostrophic
  equations with inviscid geostrophic balance.
\newblock {\em Numer. Math.}, 107:669--705, 2007.

\bibitem{ShenXu_NM21}
J.~Shen and J.~Xu.
\newblock Unconditionally positivity preserving and energy dissipative schemes
  for {Poisson--Nernst--Planck} equations.
\newblock {\em Numer. Math.}, 148:671--697, 2021.

\bibitem{SWZ_CMS18}
F.~Siddiqua, Z.~Wang, and S.~Zhou.
\newblock A modified {P}oisson--{N}ernst--{P}lanck model with excluded volume
  effect: theory and numerical implementation.
\newblock {\em Commun. Math. Sci.}, 16 (1):251--271, 2018.

\bibitem{Slotboom1973}
J.~Slotboom.
\newblock Computer-aided two-dimensional analysis of bipolar transistors.
\newblock {\em IEEE Transactions on Electron Devices}, 20(8):669--679, 1973.

\bibitem{SunSunZhengLin16}
Y.~Sun, P.~Sun, B.~Zheng, and G.~Lin.
\newblock Error analysis of finite element method for
  {P}oisson--{N}ernst--{P}lanck equations.
\newblock {\em J. Comput. Appl. Math.}, 301:28--43, 2016.

\bibitem{WL00}
C.~Wang and J.~Liu.
\newblock Convergence of gauge method for incompressible flow.
\newblock {\em Math. Comp.}, 69:1385--1407, 2000.

\bibitem{WL02}
C.~Wang and J.~Liu.
\newblock Analysis of finite difference schemes for unsteady {Navier-Stokes}
  equations in vorticity formulation.
\newblock {\em Numer. Math.}, 91:543--576, 2002.

\bibitem{WLJ04}
C.~Wang, J.~Liu, and H.~Johnston.
\newblock Analysis of a fourth order finite difference method for
  incompressible {Boussinesq} equations.
\newblock {\em Numer. Math.}, 97:555--594, 2004.

\bibitem{wang11a}
C.~Wang and S.~M. Wise.
\newblock An energy stable and convergent finite-difference scheme for the
  modified phase field crystal equation.
\newblock {\em SIAM J. Numer. Anal.}, 49:945--969, 2011.

\bibitem{wise09a}
S.~M. Wise, C.~Wang, and J.~Lowengrub.
\newblock An energy stable and convergent finite-difference scheme for the
  phase field crystal equation.
\newblock {\em SIAM J. Numer. Anal.}, 47:2269--2288, 2009.

\bibitem{Wu:JCP:2002}
Y.~Yu and J.~Wu.
\newblock Structures of hard-sphere fluids from a modified fundamental-measure
  theory.
\newblock {\em J. Chem. Phys.}, 117(22):10156--10164, 2002.

\bibitem{IonChanel_HandbookCRC15}
J.~Zheng and M.~Trudeau.
\newblock {\em Handbook of ion channels}.
\newblock CRC Press, 2015.

\bibitem{ZCW11}
Q.~Zheng, D.~Chen, and G.~Wei.
\newblock Second-order {P}oisson-{N}ernst-{P}lanck solver for ion channel
  transport.
\newblock {\em J. Comput. Phys.}, 230(13):5239--5262, 2011.

\bibitem{ZhouWangLi_PRE11}
S.~Zhou, Z.~Wang, and B.~Li.
\newblock Mean-field description of ionic size effects with non-uniform ionic
  sizes: {A} numerical approach.
\newblock {\em Phys. Rev. E}, 84:021901, 2011.

\end{thebibliography}

\end{document}